\documentclass[reqno,a4paper]{amsart}
\usepackage[fontsize=13pt]{scrextend}
\usepackage{amsthm, amsmath,amsfonts, amssymb, enumerate}
\usepackage{caption}
\usepackage{subcaption}
\usepackage{hyperref}
\usepackage{amsrefs}
\usepackage[pdftex]{graphicx}
\usepackage{color}

\newcommand{\R}{\mathbb R}
\newcommand{\E}{E}

\newcommand{\K}{\mathcal K}
\newcommand{\W}{\mathcal W}
\newcommand{\Y}{\mathcal Y}

\newcommand{\Dr}{\mathbb A_{\mathrm{ref}}}
\newcommand{\Da}{\mathbb A_{\mathrm{amb}}}
\newcommand{\acts}{\curvearrowright}

\numberwithin{equation}{section}
\theoremstyle{plain}

\newtheorem{theorem}{Theorem}[section]
\newtheorem*{mainthm*}{Main Theorem}
\newtheorem{lem}[theorem]{Lemma}
\newtheorem{prop}[theorem]{Proposition}
\newtheorem{thm}[theorem]{Theorem}
\newtheorem{cor}[theorem]{Corollary}

\theoremstyle{definition}
\newtheorem{defn}[theorem]{Definition}
\newtheorem{remark}[theorem]{Remark}
\newtheorem{exa}[theorem]{Example}

\begin{document}

\title{Dihedral twists in the Twist Conjecture}

\author[P.~Przytycki]{Piotr Przytycki$^{\dag}$}
\address{805 Sherbrooke Street West,\\Math. \& Stats.\\
                    McGill University \\
                    Montreal, Quebec, Canada H3A 0B9}
\email{piotr.przytycki@mcgill.ca}
\thanks{$\dag$ This work was partially supported by
the grant 346300 for IMPAN from the Simons Foundation and the matching 2015--2019 Polish MNiSW fund, as well as AMS, NSERC and National Science Centre, Poland UMO-2018/30/M/ST1/00668.}

\maketitle

\begin{abstract}
\noindent Under the assumption that a defining graph of a Coxeter group admits only subsequent elementary twists in $\mathbb{Z}_2$ or dihedral groups and is of type $\mathrm{FC}$, we
prove Bernhard M\"uhlherr's Twist Conjecture.
\end{abstract}

\section{Introduction}
We make progress towards verifying Bernhard M\"uhlherr's Twist Conjecture. This conjecture predicts that angle-compatible Coxeter generating sets differ by a sequence of elementary twists. By \cite{HM} and \cite{MM} the Twist Conjecture solves the Isomorphism Problem for Coxeter groups.

\begin{mainthm*}
Let $S$ be a Coxeter generating set of type $\mathrm{FC}$ angle-compatible with a Coxeter generating set $S'$. Suppose that any Coxeter generating set twist equivalent to $S$ admits only elementary twists in~$\mathbb{Z}_2$ or the dihedral groups. Then $S$ is twist equivalent to~$S'$.
\end{mainthm*}

Note that in the case where $S$ does not admit any elementary twist, we proved the Twist Conjecture in \cite{CP}. The bookkeeping in the proof was much simpler assuming $S$ is of FC type, but we managed to remove that assumption in the last section of \cite{CP}. In \cite{HP} we kept that assumption and we confirmed the Twist Conjecture in the case where we allow elementary twists but require they are all in~$\mathbb{Z}_2$. In the current article we follow this strategy amounting to allowing gradually elementary twists in larger groups. For more historical background, see our previous paper~\cite{HP}. Note that we will be invoking some basic lemmas from~\cite{HP}, but not its Main Theorem.

\textbf{Definitions.} A \emph{Coxeter generating set} $S$ of a group $W$ is a set such that $(W,S)$ is a Coxeter system. This means that $S$ generates $W$ subject only to relations of the form $s^2=1$ for $s\in S$ and $(st)^{m_{st}}=1$, where $m_{st}=m_{ts}\geq 2$ for $s\neq t\in S$ (possibly there is no relation between $s$ and $t$, and then we put by convention $m_{st}=\infty$).
An
\emph{$S$-reflection} (or a \emph{reflection}, if the dependence
on $S$ does not need to be emphasised) is an
element of $W$ conjugate to some element of $S$.
We say that $S$ is \emph{reflection-compatible} with another Coxeter generating set~$S'$ if every $S$-reflection is an $S'$-reflection. Furthermore, $S$ is \emph{angle-compatible} with $S'$ if for every $s, t \in S$ with $\langle s,t \rangle$ finite, the set $\{s,t\}$ is conjugate to some $\{s', t'\} \subset S'$.

We call a subset $J\subseteq S$ \emph{spherical} if $\langle J\rangle$ is finite. If $J$ is
spherical, let $w_J$ denote the longest element of~$\langle J\rangle$. We say
that two elements $s\neq t\in S$ are \emph{adjacent} if $\{s,t\}$ is
spherical. This gives rise to a graph whose vertices are the elements of $S$ and whose edges (labelled by $m_{st}$) correspond to adjacent pairs in $S$. This graph is called the \emph{defining graph} of $S$.
Occasionally, when all $m_{st}$ are finite, we will use another graph, whose vertices are still the elements of $S$, but (labelled) edges correspond to pairs of non-commuting elements of~$S$. This graph is called the \emph{Coxeter--Dynkin diagram} of $S$. Whenever we talk about adjacency of elements of $S$, we always mean adjacency in the defining graph unless otherwise specified.

Given a subset $J \subseteq S$, we denote by $J^\bot$ the
set of those elements of~$S\setminus J$ that commute with $J$. A
subset $J\subseteq S$ is \emph{irreducible} if it is not contained in
$K\cup K^\bot$ for some non-empty proper subset $K\subset J$.

\textbf{Type FC.} We say that $S$ is of \emph{type $\mathrm{FC}$} if each $J\subseteq S$ consisting of pairwise adjacent elements is spherical.

\textbf{Elemenary twists.} Let $T \subseteq S$. We say that  $C\subseteq T$ is a \emph{component of $T$},
if the subgraph induced on $C$ in the defining graph of $S$ is a connected component of the subgraph induced on $T$.
Let $J\subseteq S$ be an irreducible spherical subset. Assume that we have a nontrivial partition
$S\setminus (J\cup J^\bot)=A\sqcup B$, where each component $C$ of $S\setminus (J\cup J^\bot)$ is contained entirely in $A$ or in $B$. In other words, for all $a \in A$ and $b \in
B$, we have that $a$ and $b$ are non-adjacent. We then say that $J$ \emph{weakly separates} $S$.
The map $\tau \colon S \to W$ defined by
$$
\tau(s)= \left\{\begin{array}{ll}
s & \text{for } s \in A\cup J\cup J^\perp,\\
w_J s w_J^{-1} & \text{for } s \in B,
\end{array}\right.
$$
is called an \emph{elementary twist in $\langle J \rangle$} (see \cite[Def 4.4]{BMMN}).
Coxeter generating sets $S$ and $S'$ of $W$ are \emph{twist equivalent} if $S'$ can be obtained from $S$ by a finite sequence of elementary twists and a conjugation. We say that $S$ is \emph{$k$-rigid} if for each weakly separating $J\subset S$ we have $|J|<k$. Thus the assumption in the Main Theorem amounts to all Coxeter generating sets twist equivalent to $S$ being $3$-rigid. Note that being of type $\mathrm{FC}$ is invariant under elementary twists.

\textbf{Proof outline.} Let $\Da$ be the Davis complex for $(W,S')$, and for each reflection $r\in W$, let $\W_r$ be its wall in $\Da$. In Section~\ref{sec:prem}, following \cite{CM}, we explain that
to prove that $S$ (which might differ from the original $S$ by elementary twists) is conjugate to $S'$, we must find a `geometric' set of halfspaces for $s$ with $s\in S$. To this end, we will use `markings', introduced in \cite{CP} and discussed in Section~\ref{subsec:bases and markings}. These are triples $\mu=((s,w),m)$ with $w=j_1\cdots j_n$ where $j_i, s,m\in S$ satisfy certain conditions guaranteeing in particular $\mathcal W_s\cap w\mathcal W_m=\emptyset$. This determines a halfspace $\Phi^\mu_s$ for $s$ containing $w\mathcal W_m$. As in \cite{CP}, to prove that the set of these halfspaces is geometric, it suffices to prove that $\Phi^\mu_s$ depends only on~$s$.

Until the last section, our goal becomes to prove the following `consistency' of irreducible spherical $\{s,t\}\subset S$. Consistency means that all $\Phi^\mu_s$ with $j_1=t$ are equal, all $\Phi^\mu_t$ with $j_1=s$ are equal, and these two halfspaces form a geometric pair. To this end, we introduce the `complexity' $(\mathcal K_1(S),\mathcal K_2(S))$ of $S$ with respect to $S'$. The first entry $\mathcal K_1(S)$ is the sum of the distances in $\Da^{(1)}$ between all the pairs of residues~$C_L$ fixed by maximal spherical $L\subset S$. The second entry $\mathcal K_2(S)$ is the sum of the distances between more subtle objects. Namely, for maximal spherical $L\subset S$ let $D_L\subseteq C_L$ consist of chambers adjacent to each $\mathcal W_l$ with $l\in L$. The contribution to $\mathcal K_2(S)$ of a pair $L,I$ of maximal spherical subsets of $S$ is the distance between particular $\E_{L,I}\subseteq D_L$ and $\E_{I,L}\subseteq D_I$. Let us explain in detail what $\E_{L,I}$ is for $L$ irreducible.

First notice that then $D_L$ consists of exactly two opposite chambers.
We say that $L$ is `exposed' if $|L|\leq 2$ or $|L|=3$ and there are at least two elements of $L$ not adjacent to any element of $S\setminus (L\cup L^\perp)$. For $L$ exposed 
we set $\E_{L,I}=D_L$. Otherwise, we can predict which of the two chambers is better positioned with respect to $I$, and we set $\E_{L,I}$ to be that chamber. Namely, we choose $\E_{L,I}$ inside $\Phi^\mu_s$ for $m\in I$ and `good' $s$ and $\{s,j_1\}$. The notion of `good' is discussed in Section~\ref{subsec:relative position}. For example if $m$ adjacent to both $j_1$ and $j_2$, then $s$ and $\{s,j_1\}$ are good. We designed this notion to make $\E_{L,I}$ independent of the choice of $s,j_1$, which is proved in Sections~\ref{sec:independent} and~\ref{sec:independent2}. This allows us to define the complexity in Section~\ref{sec:complexity}. From now on we assume that the complexity of $S$ is minimal among all Coxeter generating sets twist equivalent to~$S$.

Going back to the goal of proving the consistency of $\{s,t\}$, we consider the components of $S\setminus (\{s,t\}\cup \{s,t\}^\perp)$. Using the 3-rigidity of~$S$ and `moves', we obtain that the markings $\mu=((s,tp\cdots),m)$ and $((s,t),p)$ with various $p$ in a fixed component $A$ give rise to the same $\Phi^{t,A}_{s}:=\Phi^\mu_s$. Thus to prove the consistency of $\{s,t\}$ one needs to prove that  the pair $\Phi^{t,A}_{s},\Phi^{s,A}_{t}$ is geometric (which we call the `self-compatibility' of $A$), and that $\Phi^{t,A}_{s}=\Phi^{t,B}_{s}$ and $\Phi^{s,A}_{t}=\Phi^{s,B}_{t}$ for every other component $B$ (which we call the `compatibility' of $A$ and $B$). We gradually show that in the subsequent  sections. There we call $A$ `small' if all the elements of $A$ are adjacent to both $s,t$; we call $A$ `big' otherwise. We call (small or big) $A$ `exposed' if there is an exposed $L\supset \{s,t\}$ intersecting $A$.

In Section~\ref{sec:proof} we prove that small components are self-compatible, and that each exposed component is self-compatible and compatible with any other component. This is done using various elementary twists provided by an exposed $L$, which allow to turn $\E_{L,I}=D_L$ `towards'~$C_I$ and decrease~$\mathcal K_2$ in the case of incompatibility. 

In Section~\ref{sec:big}, which is the central part of the article, we prove the compatibility of big components. 
For example, for the big components of $S\setminus (\{s,t\}\cup \{s,t\}^\perp)$ with $m_{st}=3$, we partition $S\setminus (\{s,t\}\cup \{s,t\}^\perp)$ into $A\sqcup B$, where $A$ is the union of components $A_i$ with one $\Phi^{t,A_i}_{s}$, and $B$ is the union of components $B_i$ with the other $\Phi^{t,B_i}_{s}$. We then prove that the elementary twist $\tau$ sending each element $b\in B$ to $tstbtst$ and fixing the other elements of $S$ decreases $\mathcal K_1$. In the case where $m_{st}>3$, a crucial concept is that of `peripherality', which picks out the `least' inconsistent $\{s,t\}$ and allows to decrease $\mathcal K_1$ using a `folding'. Finally, in Section~\ref{sec:small} we prove the self-compatibility of big components, and their compatibility with small ones.

Having established the consistency of doubles $\{s,t\}$, it is not hard to prove that $\Phi^\mu_s$ depends only on~$s$ (which as we explained implies the Main Theorem), following a simplified version of the main argument of~\cite{HP}, which we present in Section~\ref{sec:consistent}.

\textbf{Reading the article.} Upon a first reading, we recommend to ignore~$\mathcal K_2$ and assume that there are no small or exposed components. This means skipping Sections~\ref{subsec:relative position}--\ref{sec:proof} except for the definition of~$\mathcal K_1$ and the ones in Section~\ref{sec:proof}, and focusing on understanding the details of Section~\ref{sec:big}. After that, it should become clear that to treat small components it is not enough to use only $\mathcal K_1$, which motivates the introduction of~$\mathcal K_2$ with all its technical aspects.

Let us also mention that our construction of a `folding' in Section~\ref{sec:big} for $m_{st}=4$ agrees with the construction in the article of Weigel \cite[Fig~1]{W}. His assumptions on the defining graph do not allow for irreducible spherical subsets~$L$ with $|L|>2$, so he does not need to discuss small components or~$\mathcal K_2$. However, our Main Theorem does not imply the Main Theorem of \cite{W} since Weigel allows for some subsets of $S$ that violate FC.

\textbf{Acknowledgements.} We thank Pierre-Emmanuel Caprace, with whom we designed a large bulk of the strategy executed in this paper, including the main ideas in Section~\ref{sec:big}. We also thank Jingyin Huang for many long discussions and for designing together Sections~\ref{subsec:relative position}--\ref{sec:independent2}. We thank the referee for many helpful suggestions.

\section{Preliminaries}
\label{sec:prem}

\subsection{Davis complex}
\label{subsec:Davis}

Let $\mathbb A$ be the \emph{Davis complex} of a Coxeter system~$(W,S)$ (see~\cite[\S7.3]{Davis} for a precise definition). The 1-skeleton of~$\mathbb A$ is the Cayley graph of $(W,S)$ with vertex set ~$W$ and a single edge spanned on $\{w,ws\}$ for each $w\in W, s\in S$. Higher dimensional cells of $\mathbb A$ are spanned on left cosets in $W$ of remaining finite~$\langle J\rangle$. The left action of~$W$ on itself extends to the action on $\mathbb A$.

A \emph{chamber} is a vertex of $\mathbb A$. Collections of chambers corresponding to cosets $w\langle J\rangle$ are called $J$-\emph{residues} of $\mathbb A$.  A \emph{gallery} is an edge-path in $\mathbb A$. For two chambers $c_1,c_2\in\mathbb A$, we define their \emph{gallery distance}, denoted by $d(c_1,c_2)$, to be the length of a shortest gallery from $c_1$ to~$c_2$.

Let $r\in W$ be an $S$-reflection. The fixed point set of the action of $r$ on~$\mathbb A$ is called its \emph{wall} $\Y_r$. The wall $\Y_r$ determines $r$ uniquely. Moreover, $\Y_r$ separates~$\mathbb A$ into two connected components, which are called \emph{halfspaces (for $r$)}.
If a non-empty subset $K\subset\mathbb A$ is contained in a single halfspace, then $\Phi(\Y_r,K)$ denotes this halfspace. An edge of $\mathbb A$ crossed by $\Y_r$ is \emph{dual} to~$\Y_r$. A chamber is \emph{incident} to $\Y_r$ if it is an endpoint of an edge dual to $\Y_r$. The \emph{distance} of a chamber~$c$ to~$\Y_r$, denoted by $d(c,\Y_r)$, is the minimal gallery distance from $c$ to a chamber incident to $\Y_r$.

\subsection{Geometric set of reflections}
\label{sec:subs}
Let $(W,S)$ be a Coxeter system. Let $\Dr$ be the Davis complex for $(W,S)$ (`ref' stands for `reference complex'). For each reflection $r$, let~$\Y_r$ be its wall in $\Dr$.
Suppose that $S$ is \textbf{angle-compatible} with another Coxeter generating set $S'$. Let $\Da$ be the Davis complex for $(W,S')$ (`amb' stands for `ambient complex'). For each reflection $r$, let $\W_r$ be its wall in $\Da$. Let $P\subseteq S$.

\begin{defn}
\label{def:geom}
Let $\{\Phi_p\}_{p\in P}$ be a collection of halfspaces of $\Da$ for~$p\in P$. The collection $\{\Phi_p\}_{p\in P}$ is \emph{$2$-geometric} if for any pair $p,r\in P$, the set $\Phi_{p}\cap\Phi_{r}\cap \Da^{(0)}$ is a fundamental domain for the action of $\langle p,r\rangle$ on $\Da^{(0)}$.
The collection $\{\Phi_p\}_{p\in P}$ is \emph{geometric} if additionally $F=\bigcap_{p\in P}\Phi_p\cap\Da^{(0)}$ is non-empty.
The set $P$ is \emph{$2$-geometric} if there exists a $2$-geometric collection of halfspaces $\{\Phi_p\}_{p\in P}$. \end{defn}

\begin{thm}[{\cite[Thm 4.2]{CM}}]
\label{thm:geometric}
If $\{\Phi_p\}_{p\in P}$ is $2$-geometric, then after possibly replacing each $\Phi_p$ by opposite halfspace, the collection $\{\Phi_p\}_{p\in P}$ is geometric.
\end{thm}

Theorem~\ref{thm:geometric} justifies calling $2$-geometric $P$ \emph{geometric} for simplicity. We call $F$ as above a \emph{geometric fundamental domain for $P$}, since by~\cite{Hee} (see also \cite[Thm~1.2]{HRT} and \cite[Fact~1.6]{CM}), we have:

\begin{prop}
\label{prop:geometric}
If $P$ is geometric, then $F$ is a fundamental domain for the action of $\langle P \rangle$ on $\Da^{(0)}$, and for each $p\in P$ there is a chamber in $F$ incident to $\W_p$. In particular, if $P=S$, then $S$ is conjugate to~$S'$.
\end{prop}

\begin{cor}[{\cite[Cor 2.6]{HP}}]
	\label{cor:spherical conjugate}
Let $J\subseteq S$ be spherical. Then $J$ is conjugate to a spherical $J'\subseteq S'$. In particular, $J$ is geometric, and if it is irreducible, there exist exactly two geometric fundamental domains for $J$.
\end{cor}

We will need the following compatibility result.

\begin{lem}
	\label{lem:change of base}
Let $J\subset S$ be irreducible spherical, and let $r_1,r_2\in S\setminus J$ with $J\cup\{r_1,r_2\}$ geometric. Let $\W_1$ and $\W_2$ be walls of $\Da$ fixed by some reflections in~$\langle J\rangle$ and satisfying $\W_i\cap\W_{r_i}=\emptyset$ for $i=1,2$.
Let $\Delta_1, \Delta_2$ be the geometric fundamental domains for $J$ satisfying
$\Phi(\W_i,\Delta_i)=\Phi(\W_i,\W_{r_i})$ for $i=1,2$. Then $\Delta_1=\Delta_2$.
\end{lem}

\begin{proof}
Let $F\subset \Da^{(0)}$ be the geometric fundamental domain for $J\cup \{r_1,r_2\}$. By Proposition~\ref{prop:geometric}, for $i=1,2$, there is chamber $x_i\in F$ incident to $\W_{r_i}$. Let $\Delta$ be the geometric fundamental domain for $J$ containing~$F$. Then for $i=1,2,$ we have $\Phi(\W_i,\W_{r_i})=\Phi(\W_i,x_i)=\Phi(\W_i,F)=\Phi(\W_i,\Delta)$, and so $\Delta_i=\Delta$.
\end{proof}

We close with the following result, which is \cite[Lem~5.4]{HP}. Note that we assumed there that $\mathcal W=\mathcal W_r$ for some $r\in S$, but the proof works word for word without that assumption.

\begin{lem}
\label{sublem:extra}
Let $\{j_1,j_2\}\subset S$ be irreducible spherical. Suppose that a wall $\mathcal W$ in $\Da$ is disjoint from
both $\W_{j_2}$ and $j_1\W_{j_2}$, and we have $\Phi(\W_{j_2},\W)=\Phi(\W_{j_2},j_1\W)$.
Let $F$ be a geometric fundamental domain for $\{j_1,j_2\}$.
Then $\mathcal W$ is disjoint from $j_2\W_{j_1}$ and we have $\Phi(\W_{j_2},\W)=\Phi(\W_{j_2},F)$ if and only if $\Phi(\W_{j_1},j_2\W)=\Phi(\W_{j_1},F)$.
\end{lem}

We have the following immediate consequence, which is a variant of \cite[Lem~5.1]{CP}.

\begin{cor}
\label{cor:extra}
Let $\{j_1,j_2\}\subset S$ be irreducible spherical. Suppose that a wall $\mathcal W$ in $\Da$ is disjoint from
both $\W_{j_2}$ and $j_1\W_{j_2}$, and intersects~$\W_{j_1}$. Let $F$ be a geometric fundamental domain for $\{j_1,j_2\}$.
Then $\mathcal W$ is disjoint from $j_2\W_{j_1}$ and we have $\Phi(\W_{j_2},\W)=\Phi(\W_{j_2},F)$ if and only if $\Phi(\W_{j_1},j_2\W)=\Phi(\W_{j_1},F)$.
\end{cor}

\section{Bases and markings}
\label{subsec:bases and markings}
Henceforth, in the entire article we assume that \textbf{$S$ is irreducible, not spherical, and of type $\mathrm{FC}$}. (The reducible case easily follows from the irreducible.)

In this section we recall several central notions from \cite{CP}.
Let $W,S,\Dr,\Y_r$ (and later $S',\Da, \W_r$) be as in Section~\ref{sec:subs}.  Let $c_0$ be the identity chamber of~$\Dr$.

\subsection{Bases}
\begin{defn}
	\label{domain} A \emph{base} is a pair
	$(s,w)$ with \emph{core} $s\in S$ and $w\in W$ satisfying
	\begin{enumerate}[(i)]
	\item
		$w=j_1\cdots j_n$, where $n\geq 0$, and $j_i\in S$,
	\item
		$d(w.c_0,\mathcal{Y}_{s})=n$,
	\item
		the \emph{support} $J=\{s,j_1,\ldots,j_n\}$ is spherical.
	\end{enumerate}
	\end{defn}

Note that this agrees with \cite[Def 3.1]{CP}. Indeed, Condition~(ii) from \cite[Def 3.1]{CP} saying that every wall that separates $w.c_0$ from $c_0$ intersects~$\mathcal{Y}_s$ follows immediately from our Condition~(iii). On the other hand, our Condition~(iii) follows from \cite[Lem 3.5]{CP} since $S$ is of type FC. Note also that our Condition~(ii) implies that $J$ is irreducible.
A base is \emph{simple} if $s$ and all $j_i$ are distinct.
In \cite[Lem 3.7]{CP} and the paragraph preceding it, we established the following.

\begin{remark}
\label{rem:unique}
If $J\subset S$ is irreducible spherical and $s\in J$, then there
is a unique simple base $(s,w)$ with support $J$ and core $s$. We have $w=j_1\cdots j_n$ for any ordering of the elements of $J\setminus \{s\}$ into a sequence $(j_i)$ with each $\{s,j_1,\ldots, j_i\}$ irreducible. We often denote that base $(s,w)$ by $(s,J)$.
\end{remark}

The following result is a straightforward generalisation of \cite[Lem~3.3]{HP}, where a base was assumed to be simple.

\begin{lem}
	\label{lem:spherical same side}
Let $J\subset S$ be irreducible spherical, and let $F$ be a geometric fundamental domain for $J$. Then for any base $(s,w)$ with support~$J$ we have $\Phi(\W_s,F)=\Phi(\W_s,wF)$.
\end{lem}

\subsection{Markings}

\begin{defn}
\label{marking} A \emph{marking} is a pair $\mu=((s,w),m)$, where $(s,w)$ is a base with support $J$ and where the \emph{marker} $m\in S$ is such that $J\cup \{m\}$ is not spherical. The \emph{core} and the \emph{support} of the marking $\mu$ are the core and the support of its base. We say that $\mu$ is \emph{simple}, if its base is simple.
\end{defn}

Our definition of a marking agrees with the notion of a \emph{complete marking} from \cite[Def~3.8]{CP}. To see that, note that since $S$ if of type FC, $m$~is not adjacent to some element of $J$ and hence
by \cite[Rem~3.2(ii)]{CP} we have that $w\mathcal Y_m$ is disjoint from $\mathcal Y_s$. We decided to drop the term `complete' since we will not be discussing any other markings in this article. Similarly, our definition of a simple marking agrees with the notion of a \emph{good marking} from \cite[Def~3.13]{CP}, since by FC there are no semicomplete markings described in \cite[Def~3.11]{CP}.

\begin{remark}[{\cite[Rem~3.5]{HP}}]
\label{rem:find markings}
For each $s\in I\subset S$ with $I$ irreducible spherical, there exists a simple marking with support containing $I$ and core $s$.
\end{remark}

\begin{defn}
\label{halfspace}
Let $\mu=((s,w),m)$ be a marking. Since $w\mathcal Y_m$ is disjoint from $\mathcal Y_s$, the element $wmw^{-1}s$ is of infinite order, and hence also $w\mathcal W_m$ is disjoint from $\mathcal W_s$. We define
$\Phi_s^\mu=\Phi(\mathcal{W}_s, w\mathcal{W}_m)$.
\end{defn}

\begin{prop}[{\cite[Prop 5.2]{CP}}]
\label{consistence}
Let $s_1,s_2\in S$. Suppose that for each $i=1,2$, any simple marking $\mu$ with core~$s_i$ gives rise to the same
$\Phi_{s_i}=\Phi_{s_i}^{\mu}$. Then the pair $\Phi_{s_1}, \Phi_{s_2}$
is geometric.
\end{prop}

We summarise Proposition~\ref{consistence},
Theorem~\ref{thm:geometric}, and Proposition~\ref{prop:geometric} in the following.

\begin{cor}[{\cite[Cor 3.8]{HP}}]
	\label{cor:geometric criterion}
	If for each $s\in S$ any simple marking $\mu$ with core $s$ gives rise to the same $\Phi_s^{\mu}$, then $S$ is conjugate to $S'$.
\end{cor}

\subsection {Moves}
\begin{defn}
	\label{def:move}
	Let $((s,w),m), ((s,w'),m')$ be markings with common core. We say that they are
	related by \emph{move}
	\begin{enumerate}
		\item[(M1)] if $w=w'$, and the markers $m$ and $m'$ are adjacent;
		\item[(M2)] if there is $j\in S$ such that $w=w'j$ and
		moreover $m$ equals $m'$ and is adjacent to $j$.
	\end{enumerate}
We will write $((s,w),m)\sim((s,w'),m')$ if there is a finite sequence of moves M1 or M2 that brings  $((s,w),m)$ to $((s,w'),m')$.
\end{defn}

The following is a special case of \cite[Lem 4.2]{CP}.
\begin{lem}
	\label{moves do not change halfspace} If markings $\mu$ and $\mu'$ with common core $s$
	are related by move $\mathrm{M1}$ or~$\mathrm{M2}$, then $\Phi_s^\mu=\Phi_s^{\mu'}$.
	\end{lem}

We have a straightforward generalisation of \cite[Prop~4.3]{HP}.

\begin{prop}
	\label{prop:compactible1}
Let $(s,w)$ be a base with support $I$. Suppose that no irreducible spherical $I'\supsetneq I$ weakly separates $S$.
Let $\mu_1=((s,ww_1),m_1)$ and $\mu_2=((s,ww_2),m_2)$ be markings with supports $J_1,J_2$, where each of $w_i$ is a product of distinct elements of $J_i\setminus I$. Moreover, for $i=1,2$ define $K_i=J_i\setminus (I\cup I^{\perp})$ when $I\subsetneq J_i$, and $K_i=\{m_i\}$ when $J_i=I$. Suppose that $K_1$ and $K_2$ are in the same component of $S\setminus(I\cup I^{\perp})$. Then $\mu_1\sim\mu_2$. Consequently $\Phi_s^{\mu_1}=\Phi_s^{\mu_2}$.
\end{prop}

\subsection {Applications to $3$-rigid $S$}

We start with choosing the notation for the $K_i$ above in the case where $\mu_i$ is simple.

\begin{defn}
\label{def:K}
Let $\mu=((s,J),m)$ be a simple marking with the base $(s,J)$ defined in Remark~\ref{rem:unique}. Let $I\subseteq J$ be irreducible with $s\in I$. Then we denote by $K^\mu_I$
the set $J\setminus(I\cup I^\perp)$ if $J\neq I$, and the set $\{m\}$ otherwise. We simplify the notation $K^\mu_{\{s\}}$ to $K^\mu_s$ etc.

Let $\{s,t\}\subset S$ be irreducible spherical. For each component $A$ of $S\setminus(\{s,t\}\cup \{s,t\}^{\perp})$, taking $k\in A$ and $\mu=((s,t),k)$ (if it is a marking) or applying Remark~\ref{rem:find markings} to $I=\{s,t,k\}$, we obtain a simple marking~$\mu$ with support containing $t$, core~$s$, and $K^\mu_{s,t}\subseteq A$. If $S$ is $3$-rigid, then
by Proposition~\ref{prop:compactible1} if we have $\mu'$ with $K^{\mu'}_{s,t}\subseteq A$, then $\Phi^{\mu}_s=\Phi^{\mu'}_s$. Thus each component $A$ of $S\setminus(\{s,t\}\cup \{s,t\}^{\perp})$ determines a halfspace $\Phi^{t,A}_{s}:=\Phi^{\mu}_s$ for $s$.
\end{defn}

The following is another variation on \cite[Prop~4.3]{HP}.

\begin{prop}
\label{prop:same side2}
Suppose that $S$
is $3$-rigid. Let $\mu_1,\mu_2$ be simple markings with common core $s$.
Suppose that $K^{\mu_1}_s\cap K^{\mu_2}_s=\emptyset$ and that there is an embedded path $\omega$ in the defining graph of $S$ outside $s\cup s^{\perp}$ starting in $k_1\in K^{\mu_1}_s$ and ending in $k_2\in K^{\mu_2}_s$ such that \begin{enumerate}[(i)]
\item
for any vertex $k\neq k_1,k_2$ of $\omega$ adjacent to $s$ and any simple markings $\nu_1,\nu_2$ with supports containing $k$ and core $s$ we have $\Phi_{s}^{\nu_1}=\Phi_{s}^{\nu_2}$.
\item
if $k_1$ is adjacent to $s$, then $K^{\mu_1}_{s,k_1}$ lies in the same component of $S\setminus (\{s,k_1\}\cup \{s\cup k_1\}^{\perp})$ as $k_2$, and
\item
condition (ii) holds with indices $1$ and $2$ interchanged.
\end{enumerate}
Then $\Phi_s^{\mu_1}=\Phi_s^{\mu_2}$.
\end{prop}

\begin{proof}
We proceed by induction of the length of $\omega$. Consider first the case where $k_1$ and $k_2$ are adjacent. If neither $k_1$ nor $k_2$ is adjacent to~$s$, then $\mu_1=((s,\emptyset), k_1), \mu_2=((s,\emptyset),k_2)$, and therefore $\mu_1\sim \mu_2$ by move~M1. If exactly one of $k_1,k_2$, say $k_1$, is adjacent to $s$, then let $\mu=((s,k_1),k_2)$. We have $\mu\sim \mu_2$ by move~M2. Moreover, $\mu_1\sim \mu$ by condition (ii) and Proposition~\ref{prop:compactible1}. If both $k_1,k_2$ are adjacent to~$s$, then let $\mu$ be a simple marking with
support containing $k_1,k_2$, and core $s$, which exists by Remark~\ref{rem:find markings} and FC. By conditions~(ii) and~(iii) and Proposition~\ref{prop:compactible1} we have $\mu_1\sim \mu\sim \mu_2$. Thus $\Phi_s^{\mu_1}=\Phi_s^{\mu_2}$ by Lemma~\ref{moves do not change halfspace}.

Now consider the case where $k_1$ and $k_2$ are not adjacent. Let $k$ be a vertex of $\omega$ distinct from $k_1,k_2$. Note that
if $k_1$ is adjacent to $s$, then all the vertices of $\omega$ except for $k_1$ are contained in $S\setminus (\{s,k_1\}\cup \{s\cup k_1\}^{\perp})$, and thus condition (ii) holds with $k_2$ replaced with $k$.

First suppose that $k$ is not adjacent to $s$. Then by the previous paragraph the pair of markings $\mu_1, \nu=((s,\emptyset),k)$ satisfies the hypotheses of the proposition and thus by induction we have $\Phi_s^{\mu_1}=\Phi_s^{\nu}$. Analogously, $\Phi_s^{\nu}=\Phi_s^{\mu_2}$. Finally, suppose that $k$ is adjacent to $s$.
For $i=1,2$ let $\nu_i$ be a simple marking with support containing $k$ and core $s$ such that $K^{\nu_i}_{s,k}$ lies in the same component of $S\setminus (\{s,k\}\cup \{s\cup k\}^{\perp})$ as $k_i$, which exists by Remark~\ref{rem:find markings}. As before, by induction we have $\Phi_s^{\mu_i}=\Phi_s^{\nu_i}$. Furthermore, $\Phi_s^{\nu_1}=\Phi_s^{\nu_2}$ by condition (i).
\end{proof}


\section{Good pairs}
\label{subsec:relative position}
Let $(W,S)$ be a Coxeter system. \textbf{Throughout the remaining part of the article, we will assume that all Coxeter generating sets twist equivalent to $S$ are 3-rigid}.

The following notion of a good element $t$ varies slightly from the one in \cite{HP}, where we allowed $r$ to be adjacent to $t$.

\begin{defn}
\label{def:good}
Let $L\subset S$ be irreducible spherical and let $r\in S$. An element \emph{$t\in L$ is good with respect to $r$}, if
\begin{itemize}
\item
$r\neq t$ and $r$ is not adjacent to $t$, and
\item
$L\setminus (t\cup t^{\perp})$ is non-empty and in the same component of $S\setminus(t\cup t^{\perp})$ as~$r$.
\end{itemize}
Note that being good depends on $L$. However, we often write shortly `$t$ is good with respect to $r$' (or even just `$t$ is good'), if $L$ (and $r$) are fixed. From the first bullet it follows that $r\in S\setminus L$.

A non-commuting pair \emph{$\{s,t\}\subset L$ is good with respect to $r$}, if
\begin{itemize}
\item
$\{s,t,r\}$ is not spherical, and
\item
$L\setminus (\{s,t\}\cup \{s,t\}^{\perp})$ is non-empty and in the same component of $S\setminus(\{s,t\}\cup \{s,t\}^{\perp})$ as $r$.
\end{itemize}
\end{defn}

The following lemma and its corollary exceptionally do not require the $3$-rigidity assumption on $S$.

\begin{lem}
\label{lem:stnotgood}
Let $\{s,t\}\subset S$ be spherical irreducible, and let $r\in S$ with $\{s,t,r\}$ not spherical.
Suppose that $s\in L=\{s,t\}$ is not good with respect to $r$ and $t\in L$ is not good with respect to $r$. Then $r$ lies in a component of $S\setminus (\{s,t\}\cup \{s,t\}^\perp)$ that has no element adjacent to $s$ or $t$.
\end{lem}

\begin{proof} By FC we have that $r$ is not adjacent to $s$ or not adjacent to~$t$. If $r$ is not adjacent to, say, $s$, then $r$ is not adjacent to $t$ (since $s$ is not good). For contradiction, suppose that $r$ lies in a component of $S\setminus (\{s,t\}\cup \{s,t\}^\perp)$ that has an element adjacent to $s$ or $t$. Let $\omega=r\cdots k$ be a minimal length path in the defining graph of $S$ outside $\{s,t\}\cup \{s,t\}^\perp$ ending with a vertex $k$ adjacent to $s$ or $t$, say $t$. Since $s$ is not good, there is a vertex $k'$ of $\omega$ that lies in $s^\perp$. By the minimality of $\omega$, we have $k'=k$, and hence $k$ is also adjacent to $s$. Analogously, since $t$ is not good, we have $k\in t^\perp$. Consequently, $k\in \{s,t\}^\perp$, which is a contradiction.
\end{proof}

Lemma~\ref{lem:stnotgood} immediately implies the following.

\begin{cor}
	\label{lem:good single}
Let $L\subset S$ be irreducible spherical and let $\{s,t\}\subset L$ be a non-commuting pair. Let $r\in S\setminus L$. If $\{s,t\}$ is good with respect to $r$, then $s$ or $t$ is good with respect to $r$.
\end{cor}

\begin{lem}
	\label{lem:good pair jump 1}
Let $L\subset S$ be irreducible spherical and let $r\in S$. Let $s,t,p$ be consecutive vertices in the Coxeter--Dynkin diagram of $L$ with $\{s,t,p,r\}$ not spherical. If $\{s,t\}$ is not good with respect to $r$ and $\{t,p\}$ is not good with respect to~$r$, then none of the elements in $S\setminus(\{s,t,p\}\cup \{s,t,p\}^{\perp})$ are adjacent to $s$ or $p$.
\end{lem}

\begin{proof} If $r$ is not adjacent to, say, $s$ or $t$, then since $\{s,t\}$ is not good with respect to $r$, we have that $r$ is also not adjacent to $p$. For contradiction, suppose that an element in $S\setminus(\{s,t,p\}\cup \{s,t,p\}^{\perp})$ is adjacent to $s$ or $p$. Since $S$ is 3-rigid, we have a minimal length path $\omega=r\cdots k$ in the defining graph of $S$ outside $\{s,t,p\}\cup \{s,t,p\}^{\perp}$ with $k$ adjacent to $s$ or $p$, say $p$. Since $\{s,t\}$ is not good, a vertex $k'$ of $\omega$ lies in $\{s,t\}^\perp$. Then $k'=k$ by the minimality of $\omega$. Thus $k\in \{s,t\}^\perp$. Analogously, since  $\{t,p\}$ is not good, a vertex of $\omega$ lies in $\{t,p\}^\perp$ giving $k\in\{t,p\}^\perp$. Thus $k\in\{s,t,p\}^\perp$, which is a contradiction.
\end{proof}


\begin{cor}
	\label{cor:good pair jump}
Let $L\subset S$ be irreducible spherical with $|L|\geq 4$ and let $r\in S$. Let $s,t,p$ be consecutive vertices in the Coxeter--Dynkin diagram of~$L$ with $\{s,t,p,r\}$ not spherical. Then at least one of $\{s,t\}$ or $\{t,p\}$ is good with respect to~$r$.
\end{cor}

\begin{proof}
Let $q\in L\setminus (\{s,t,p\}\cup \{s,t,p\}^\perp)$. If $\{s,t\}$ is not good with respect to $r$ and $\{t,p\}$ is not good with respect to~$r$, then by Lemma~\ref{lem:good pair jump 1} we have that $q$ is not adjacent to $s$, which is a contradiction.
\end{proof}

We have also the following variant of Lemma~\ref{lem:good pair jump 1}.

\begin{lem}
	\label{lem:good pair jump new}
Let $L\subset S$ be irreducible spherical and let $r\in S$. Let $s,t,p$ be consecutive vertices in the Coxeter--Dynkin diagram of $L$ with $\{s,t,p,r\}$ not spherical. If $\{s,t\}$ is not good with respect to $r$ and $p$ is not good with respect to $r$, then none of the elements in $S\setminus(\{s,t,p\}\cup \{s,t,p\}^{\perp})$ are adjacent to $t$ or $p$.
\end{lem}

\begin{proof}
Again, if $r$ is not adjacent to $s$ or $t$, then since $\{s,t\}$ is not good, we have that $r$ is also not adjacent to $p$. On the other hand, if $r$ is not adjacent to $p$, then since $p$ is not good, we have that $r$ is also not adjacent to $t$. For contradiction, suppose that an element in $S\setminus(\{s,t,p\}\cup \{s,t,p\}^{\perp})$ is adjacent to $t$ or~$p$. Since $S$ is 3-rigid, we have a minimal length path $\omega=r\cdots k$ in the defining graph of $S$ outside $\{s,t,p\}\cup \{s,t,p\}^{\perp}$ with $k$ adjacent to $t$ or~$p$, say~$t$ (the other case is similar). Since $p$ is not good, a vertex $k'$ of $\omega$ lies in $p^\perp$. Then $k'=k$ by the minimality of $\omega$. Thus $k\in p^\perp$. Analogously, since $\{s,t\}$ is not good, a vertex of $\omega$ lies in $\{s,t\}^\perp$ giving $k\in\{s,t\}^\perp$. Thus $k\in\{s,t,p\}^\perp$, which is a contradiction as before.
\end{proof}

\section{Fundamental domains for good pairs}
\label{sec:independent}

Let $S,S',W,\Dr$ and $\Da$ be as in Section~\ref{subsec:bases and markings}.
In this and the following section, we fix $L\subset S$ irreducible spherical.

\begin{defn}
\label{def:Delta}
Let $\mu=((s,w),m)$ be a marking with support contained in $L$. By $\Delta^\mu$ (or $\Delta^{(s,w),m}$) we denote the geometric fundamental domain for $L$ that is contained in $\Phi^\mu_s=\Phi(\W_s,w\W_m)$. Equivalently, by Lemma~\ref{lem:spherical same side}, it is the geometric fundamental domain for $L$ that is contained in $\Phi(w^{-1}\W_s,\W_m)$.
\end{defn}

Note that $\Delta^\mu$ depends on $L$ but we suppress this in the notation.

\begin{remark}
\label{rem:moves_delta}
By Lemma~\ref{moves do not change halfspace}, for $\mu\sim\nu$ we have $\Delta^\mu=\Delta^\nu$.
\end{remark}

\textbf{In the remaining part of the section, let $\{s,t\}\subset L$ be a non-commuting pair, and let $r\in S$ with $\{s,t,r\}$ not spherical.} In particular, $((s,t),r)$ and $((t,s),r)$ are both markings. Here is the main result of the section.

\begin{prop}
\label{prop:define side}
Suppose that $\{s,t\}$ is good with respect to $r$. Suppose that both $s$ and $t$ are good with respect to $r$. Then
$\Delta^{(t,s),r}=\Delta^{(s,t),r}$.
\end{prop}

Proposition~\ref{prop:define side} makes the following notion well-defined.

\begin{defn}
Suppose that $\{s,t\}$ is good with respect to $r$. By Corollary~\ref{lem:good single}, at least one of $s,t$, say $s$, is good with respect to $r$. Then we define $\Delta^{\{s,t\},r}$ to be $\Delta^{(s,t),r}$.
\end{defn}

Proposition~\ref{prop:define side} follows immediately from the following.

\begin{prop}
\label{prop:switch base}
Suppose that both $\{s,t\}$ and $s$ are good with respect to $r$ and that there are consecutive vertices $s,t,p$ in the Coxeter--Dynkin diagram of $L$. Then $\Delta^{(t,s),r}=\Delta^{(s,t),r}$.
\end{prop}

In preparation for the proof of Proposition~\ref{prop:switch base} we discuss several lemmas. We will denote shortly
$\Delta_s=\Delta^{(s,t),r}, \Delta_t=\Delta^{(t,s),r}$.

\begin{lem}
\label{lem:figure45}
Suppose that $s,t,p$ are consecutive vertices in the Coxeter--Dynkin diagram of $L$, and $m_{st}\neq 3$.
Suppose also $((s,t),r)\sim ((s,tp),r)$ and $((t,s),r)\sim ((t,spt),r)$. Then $\Delta_t=\Delta_s$.
\end{lem}

Note that it is easy to check that for $m_{st}\neq 3$ the pair $(t,spt)$ is indeed a base. Making use of this base and its extensions is exactly the reason for which we need to discuss in this article bases that are not simple.

\begin{figure}
\begin{center}
\includegraphics[width=0.4\textwidth]{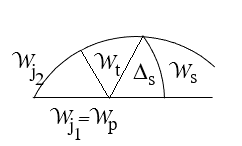}
\end{center}
\caption{}
\end{figure}

\begin{proof}
By the classification of finite Coxeter groups, we have $m_{tp}=3$.
We want to apply Lemma~\ref{sublem:extra} to the conjugate $tpSpt$, to $j_1, j_2$ the conjugates of $t,s$, so that $j_2=tpspt=tst, j_1=tptpt=p$, and to $\mathcal W=\mathcal W_r$. Since $((s,t),r)\sim ((s,tp),r)$, by Lemma~\ref{moves do not change halfspace} we have
$\Phi(\W_{j_2},\W)=\Phi(t\W_s,\W_r)=\Phi(t\W_{s},p\W_r)=\Phi(\W_{j_2},j_1\W)$, so the assumption of Lemma~\ref{sublem:extra} is satisfied. It is easy to see (Figure~1) that $\Delta_s$ lies in a geometric fundamental domain $F$ for $\{j_1,j_2\}$. By the definition of $\Delta_s$ we have $\Phi(\W_{j_2},\W)=\Phi(\W_{j_2},F)$, so by Lemma~\ref{sublem:extra} (and Lemma~\ref{lem:spherical same side}), we have $\Phi(j_2\W_{j_1},\W)=\Phi(j_2\W_{j_1},F)$. This implies $\Phi(tst\W_p,\W_r)=\Phi(tst\W_p,\Delta_s)$. Since $tst\W_p=tps\W_t$, and $((t,s),r)\sim ((t,spt),r)$,
Remark~\ref{rem:moves_delta} implies
$\Delta_t=\Delta_s$.
\end{proof}

\begin{cor}
\label{cor:figure45}
Suppose that $L=\{u,s,t,p\}\subset S$ is of type $\mathrm{F}_4$, and that $u,s,t,p$ are consecutive vertices in the Coxeter--Dynkin diagram of~$L$. Suppose also that $r$ is not adjacent to $s$ and
that both $\{s,t\}$ and $\{u,s\}$ are good with respect to $r$. Then $\Delta_t=\Delta_s$.
\end{cor}
\begin{proof} Since $\{s,t\}$ is good, by Proposition~\ref{prop:compactible1} and the $3$-rigidity of~$S$ we have $((s,t),r)\sim ((s,tp),r)$ and $((t,s),r)\sim ((t,spu),r)$. Since $\{u,s\}$ is good, there is a minimal length path $rr_1\cdots r_nt$ in the defining graph of $S$ outside $\{u,s\}\cup \{u,s\}^\perp\supset \{u,s,p\}\cup \{u,s,t,p\}^\perp$. By the classification of finite Coxeter groups $\{u,s,t,p\}$ is maximal irreducible spherical. Thus using moves~M1 and~M2 we obtain
\begin{align*}
((t,spu),r)\sim &((t,spu),r_1)\sim \cdots \sim ((t,spu),r_n)\sim\\
&((t,sput),r_n) \sim \cdots \sim ((t,sput),r_1)\sim ((t,sput),r).
\end{align*}
By the 3-rigidity of $S$ and Proposition~\ref{prop:compactible1}
we have $((t,sput),r)=((t,sptu),r)\sim ((t,spt),r)$. Thus Lemma~\ref{lem:figure45} applies.
\end{proof}

\begin{lem}
\label{lem:twisttrick}
Suppose that $L=\{u,s,t,p\}$ and that $u,s,t,p$ are consecutive vertices in the Coxeter--Dynkin diagram of~$L$. Suppose that $\{s,t,r\}$ is not spherical.
Then $\{u,s\}$ or $\{t,p\}$ is good with respect to~$r$.
\end{lem}
\begin{proof}
If $r$ is adjacent to $s$ or $t$, then the lemma follows, so suppose otherwise. By the classification of finite Coxeter groups we can assume without loss of generality $m_{us}=3$. Suppose that $\{u,s\}$ is not good. Let $\Gamma_\tau$ be the defining graph of the Coxeter generating set~$S_\tau$ obtained by the elementary twist in $\langle u,s\rangle$ that conjugates by the longest word~$w_{us}$ in $\langle u,s\rangle$ all the elements of the component $B$ of $S\setminus (\{u,s\}\cup\{u,s\}^\perp)$ containing~$r$. Note that $t\notin B$ as $\{u,s\}$ is not good. Since $S_\tau$ is $3$-rigid, $\{s,t,p\}$ does not weakly separate~$S_\tau$. Consider then a minimal length path $\omega_\tau$ in~$\Gamma_\tau$ from $w_{us}rw_{us}^{-1}$ to~$u$ outside $\{s,p,t\}\cup \{s,p,t\}^\perp$. Note that by the minimality of $\omega_\tau$ all the vertices of $\omega_\tau$ are conjugates by $w_{us}$ of the elements in $B$, except for $u$ and possibly the vertex preceding $u$, which might be in $\{u,s\}^\perp$. Thus conjugating~$\omega_\tau$ back, we obtain a path $\omega$ from $r$ to $s$ in the defining graph of $S$, contained in $B\cup \{u,s\}^\perp\cup \{s\}$ and outside $\{u,p,t\}\cup \{u,s,p,t\}^\perp$. We claim that $\omega$ lies outside $\{t,p\}^\perp$ justifying that $\{t,p\}$ is good. Otherwise, let $k$ be the first vertex of $\omega$ in~$\{t,p\}^\perp$, and let $\omega_k$ be the subpath $r\cdots k$ of $\omega$. Since $t\notin B$, the path~$\omega_k$ must have a vertex in $\{u,s\}^\perp$. By the minimality of $\omega_\tau$, this must be the vertex $k$. Thus $k$ is a vertex of $\omega$ lying in $\{u,s,p,t\}^\perp$, which is a contradiction.
\end{proof}

\begin{proof}[Proof of Proposition~\ref{prop:switch base}]
Assume $m_{st}\neq 3$, since otherwise $s\mathcal W_t=t\mathcal W_s$ and so $\Delta_t=\Delta_s$ is immediate.

Suppose first that $r$ is adjacent to one of $s,t$, say $t$.
By definition we have $\Phi(\mathcal W_t,s\mathcal W_r)=
\Phi(\mathcal W_t,\Delta_t)$. Thus applying Corollary~\ref{cor:extra} with $j_1=t,j_2=s, F\supset\Delta_t$ we obtain $\Phi(\mathcal W_s,\mathcal W_r)=\Phi(\mathcal W_s,\Delta_t)$. As $((s,\emptyset),r)\sim ((s,t),r)$ (move M2), by Remark~\ref{rem:moves_delta} we have
$\Delta_t=\Delta_s$, as required.

It remains to consider the case where $r$ is adjacent neither to $s$ nor~$t$.
Since $\{s,t\}\subset L$ is good with respect to $r$, by Proposition~\ref{prop:compactible1} and the 3-rigidity of~$S$ we have $((s,t),r)\sim ((s,tp),r)$ and $((t,s),r)\sim ((t,sp),r)$. Since $s\in L$ is good with respect to $r$, there is a path in the defining graph of $S$ from $r$ to $t$ outside $s\cup s^\perp$. If for each vertex $u\neq t$ on this path the set $\{s,t,p,u\}$ is not spherical, then using moves~M1 and~M2 we have $((t,sp),r)\sim ((t,spt),r)$ and Lemma~\ref{lem:figure45} applies. Otherwise, if $u$ is a vertex of that path adjacent to all $s,t,p$, we are in the setup of Lemma~\ref{lem:twisttrick}, with $L$ replaced by $L'=\{u,s,t,p\}$.
Thus one of $\{u,s\},\{t,p\}\subset L'$ is good with respect to $r$. Note that $L'$ and $L$ both contain $s,t,p$, so we have that $\{s,t\}\subset L'$ is still good with respect to $r$. Then, possibly after interchanging $s$ with~$t$ and $u$ with $p$, Corollary~\ref{cor:figure45} applies, with $L'$ in place of $L$ (which is of type $\mathrm{F}_4$ by $m_{st}\neq 3$ and the classification of finite Coxeter groups). Thus for $\Delta'_s,\Delta'_t$ defined as $\Delta_s,\Delta_t$, with $L'$ in place of $L$, we have $\Delta'_t=\Delta'_s$. Since by definition $\Delta'_t$ and $\Delta_t$ (and analogously $\Delta'_s$ and~$\Delta_s$) are contained in the same geometric fundamental domain for $\{s,t\}$, we have $\Delta_t=\Delta_s$, as desired.
\end{proof}

\section{Independence of fundamental domains}
\label{sec:independent2}

This section is devoted to the proof of the following.

\begin{prop}
	\label{prop:consistent}
Let $L\subset S$ be irreducible spherical and $I\subset S$ be spherical. Suppose that $\{s,t\}$ and $\{p,q\}$ are non-commuting pairs in~$L$. Let $r,r'\in I$ be such that
$\{s,t\}\subset L$ is good with respect to $r$, and $\{p,q\}\subset L$ is good with respect to $r'$.
Then $\Delta^{\{s,t\},r}=\Delta^{\{p,q\},r'}$.
\end{prop}

The key to the proof is:

\begin{lem}
\label{lem:Emu}
Let $L\subset S$ be irreducible spherical with $|L|\geq 3$ and let $r\in S$ with $L\cup \{r\}$ not spherical. Consider $s\in L$ that is not a leaf of the Coxeter--Dynkin diagram of $L$. Let $\mu=((s,L),r)$.
Then $\Delta^\mu$ does not depend on $s$.
\end{lem}

Here $(s,L)$ denotes the unique
simple base $(s,w)$ with support $L$ and core~$s$ from Remark~\ref{rem:unique}.
Before we give the proof of Lemma~\ref{lem:Emu}, we record the following.

\begin{remark}
\label{rem:Emu}
Let $L\subset S$ be irreducible spherical and let $r\in S$. Suppose that $\{s,t\}\subset L$ is a non-commuting pair that is good with respect to $r$.
\begin{enumerate}[(i)]
\item
Let $\nu=((s,t),r)$. Since we can assume that $w$ above starts with~$t$, by the $3$-rigidity of $S$, Proposition~\ref{prop:compactible1}, and Remark~\ref{rem:moves_delta} we have $\Delta^\nu=\Delta^\mu$.
\item
Since $s\in L$ is not a leaf of the Coxeter--Dynkin diagram of $L$, by Proposition~\ref{prop:switch base} we have $\Delta^{\{s,t\},r}=\Delta^\nu$.
\end{enumerate}
\end{remark}

Note that Proposition~\ref{prop:consistent} follows immediately from Remark~\ref{rem:Emu} and Lemma~\ref{lem:Emu} since $\mathcal W_r\cap \mathcal W_{r'}\neq \emptyset$ and hence $\Delta^{(s,L),r}=\Delta^{(s,L),r'}$.

\begin{proof}[Proof of Lemma~\ref{lem:Emu}] Suppose that $t\in L$ is also not a leaf of the Coxeter--Dynkin diagram of $L$. If we have $m_{st}=3$, then $s\mathcal W_{t}=t\mathcal W_{s}$ and $\Delta^{(s,L),r}=\Delta^{(t,L),r}$ follows.
It remains to analyse the situation where $u,s,t,p$ are consecutive vertices in the Coxeter--Dynkin diagram of $L$ of type $F_4$. If $\{s,t\}$ is good, then we have $\Delta^{(s,L),r}=\Delta^{(t,L),r}$ by Proposition~\ref{prop:switch base} and Remark~\ref{rem:Emu}(i).

Suppose now that $\{s,t\}$ is not good. We claim that at least one of $\{u,s\},\{t,p\}$ is good. We first establish that $r$ is not adjacent to at least one of $u,p$. Indeed, if $\{s,t,r\}$ is not spherical, then since $\{s,t\}$ is not good we have that $r$ is neither adjacent to $u$ nor to $p$. If $\{s,t,r\}$ is spherical, then since $L\cup \{r\}$ is not spherical, $r$ is not adjacent to at least one of $u,p$, say $u$. Then by Corollary~\ref{cor:good pair jump} the pair $\{u,s\}$ is good, justifying the claim. In particular, we have
\begin{equation}
\label{eq:walls1}
((s,u),r)\sim ((s,ut),r)\sim ((s,utp),r),
\end{equation}
where $\sim$ follow from the assumptions that $\{u,s\}$ is good, that $S$ is 3-rigid and from Proposition~\ref{prop:compactible1}.

Let $\nu=((s,u),r)$.
Let $H_1=\Phi(\W_t,\Delta^\nu)$ and $H_2=\Phi(\W_t,psu\W_r)$. By Remark~\ref{rem:Emu}(i),
to prove $\Delta^{(s,L),r}=\Delta^{(t,L),r}$ it suffices to show $H_1=H_2$.

Let $H=\Phi(\W_s,u\W_r)$. By Equation~(\ref{eq:walls1}) and Lemma~\ref{moves do not change halfspace}, we have $\W_r\subset uH\cap ptuH=:U$. Thus $\W_r\subset U\cap usp H_2$. On the other hand, by Lemma~\ref{lem:spherical same side}, we have $\Delta^\nu\subset U\cap usp H_1$. Hence $H_1=H_2$ by Corollary~\ref{cor:spherical conjugate} and Lemma~\ref{lem:24cell} below.
\end{proof}

\begin{lem}
	\label{lem:24cell}
Suppose $W$ is a Coxeter group of type $F_4$ with $u,s,t,p$ consecutive vertices in its Coxeter--Dynkin diagram. Consider the Tits representation $W\acts \mathbb E^4$. Let $H^+_j$ and $H^{-}_j$ be the two open halfspaces in $\mathbb E^4$ bounded by the hyperplane fixed by a generator $j$. Let $U=uH^+_s\cap ptu H^+_s$.
Then one of $U\cap usp H^+_t$ and $U\cap usp H^-_t$ is empty.
\end{lem}

\begin{proof}
The simple roots associated to $u,s,t,p$ are $\alpha_u=(1,-1,0,0), \alpha_s=(0,1,-1,0),\alpha_t=(0,0,1,0)$ and $\alpha_p=(-\frac{1}{2},-\frac{1}{2},-\frac{1}{2},-\frac{1}{2})$.
One computes directly $u\alpha_s=\alpha_u+\alpha_s, tu\alpha_s=\alpha_u+\alpha_s+2\alpha_t, ptu\alpha_s=\alpha_u+\alpha_s+2\alpha_t+2\alpha_p$. Moreover, $p\alpha_t=\alpha_t+\alpha_p,sp\alpha_t=\alpha_s+\alpha_t+\alpha_p$, and $usp\alpha_t=\alpha_u+\alpha_s+\alpha_t+\alpha_p$. Note that $u\alpha_s+ptu\alpha_s=2usp\alpha_t$. Thus for any vector $v\in \mathbb E^4$, if $\langle v, u\alpha_s\rangle>0$ and $\langle v, ptu\alpha_s\rangle>0$, then $\langle v, usp\alpha_t\rangle>0$, as desired.
\end{proof}

\section{Complexity}
\label{sec:complexity}

In this section, we introduce the complexity of the Coxeter generating set $S$ with respect to~$S'$. This extends the ideas of \cite[\S6]{HP}. We keep the setup from Section~\ref{sec:independent}. To start with, we need to distinguish particular spherical subsets.

\begin{defn}
\label{def:goodJ}
Let $L\subset S$ be irreducible spherical. $L$ is \emph{exposed} if
$|L|\leq 2$ or $|L|=3$ and there are at least two elements of $L$ not adjacent to any element of $S\setminus (L\cup L^\perp)$.
\end{defn}

Here are several criteria for identifying exposed $L$.

\begin{lem}
\label{lem:exposed_criteria}
Let $L\subset S$ be irreducible spherical, and let $r\in S$ with $L\cup \{r\}$ not spherical. Suppose that each non-commuting pair $\{s,t\}\subset L$ is not good with respect to $r$. Then $L$ is exposed.
\end{lem}

\begin{proof} Suppose $|L|\geq 3$. If for some non-commuting pair $\{s,t\}\subset L$ we have that $r$ is adjacent to both $s,t$, then let $p\in L\setminus \{s,t\}$ be non-commuting with one of $s,t$, say $t$. Since $r$ is adjacent to $s$ and $\{t,p\}$ is not good, we have that $r$ is adjacent to $p$. Proceeding in this way we get that $r$ is adjacent to all the elements of $L$, which by FC contradicts our hypothesis. Thus by Corollary~\ref{cor:good pair jump} we have $|L|=3$, and by Lemma~\ref{lem:good pair jump 1} there are at least two elements of $L$ not adjacent to any element of $S\setminus (L\cup L^\perp)$.
\end{proof}

Lemmas~\ref{lem:good pair jump 1} and~\ref{lem:good pair jump new} give also immediately the following.

\begin{lem}
\label{lem:exposed_criteria2}
Let $L\subset S$ be irreducible spherical, and $s,t,p$ be consecutive vertices in the Coxeter--Dynkin diagram of $L$. Let $r\in S$ be distinct from and not adjacent to $p$. Suppose that both of the following hold:
\begin{itemize}
\item
$\{s,t\}\subset L$ is not good with respect to $r$,
\item
$\{t,p\}\subset L$ is not good with respect to $r$ or $p\in L$ is not good with respect to $r$.
\end{itemize}
Then $L$ is exposed.
\end{lem}

We now describe particular subsets of pairs of maximal spherical residues.

\begin{defn}
Let $L\subset S$ be a maximal spherical subset. By Corollary~\ref{cor:spherical conjugate}, $\langle L\rangle$ stabilises a unique maximal cell $\sigma_L\subset\Da$. Let $C_L$ be the collection of vertices in $\sigma_L$ and let $D_L$ be the elements of $C_L$ incident to each $\W_l$ for $l\in L$.
\end{defn}

When $L$ is irreducible, then by Corollary~\ref{cor:spherical conjugate} it is easy to see that $D_L$ consists of two antipodal vertices. In general, let $L=L_1\sqcup\cdots\sqcup L_k$ be the decomposition of $L$ into maximal irreducible subsets. Let $\sigma_L=\sigma_1\times\cdots\times\sigma_k$ be the induced product decomposition of the associated cell. Then $D_L$ is a product of pairs of antipodal vertices $\{u_i,v_i\}$ for each $\sigma_i$. Let $\pi_i\colon D_L\rightarrow \{u_i,v_i\}$ be the coordinate projections.

\begin{defn}
\label{def:E}
For each ordered pair $(L,I)$ of maximal spherical subsets of $S$, we define the following subset $\E_{L,I}\subseteq D_L$. First, for each~$i=1,\ldots, k$, consider the following $\E^i_{L,I} \subseteq D_L$. If $L_i$ is exposed or $L_i\subset I$, then we take $\E^i_{L,I}=D_L$. Otherwise, since $I$ is maximal spherical, there is $r\in I$ with $L_i\cup\{r\}$ not spherical.
Moreover, by Lemma~\ref{lem:exposed_criteria}, there is $\{s,t\}\subset L_i$ that is good with respect to $r$. Then we take $\E^i_{L,I}=C_L\cap \Delta^{\{s,t\},r}$, where in Definition~\ref{def:Delta} we substitute $L$ with $L_i$.
Note that such $\E^i_{L,I}$ is contained in $D_L$ and equal $\pi^{-1}_i(u_i)$ or $\pi_i^{-1}(v_i)$. Furthermore, $\E^i_{L,I}$ does not depend on $\{s,t\}$ and $r$ by Proposition~\ref{prop:consistent}. We define $\E_{L,I}=\E^1_{L,I}\cap\cdots\cap \E^k_{L,I}$.
\end{defn}

\begin{remark}
\label{def:alternateE}
In Definition~\ref{def:E}, in the case where $L_i$ is neither exposed nor a subset of $I$, the set $\E^i_{L,I}$ can be characterised in the following alternate way that does not involve the notion of a good pair. Namely, by Remark~\ref{rem:Emu} and Lemma~\ref{lem:Emu} we have that $\Delta^{\{s,t\},r}$ is the fundamental domain for $L_i$ that is contained in $\Phi(\mathcal W_{s'}, wC_I)$, for any $s'\in S$ that is not a leaf of the Coxeter--Dynkin diagram of $L_i$ and $(s',w)$ the unique simple base with support $L_i$ and core~$s'$.
\end{remark}

\begin{defn}
	\label{def:complexity function}
We define the \emph{complexity} of $S$, denoted $\K(S)$, to be the ordered pair of numbers
	\begin{center}
		$\big(\K_1(S),\K_2(S)\big)=\Big(\sum_{L\neq I}d(C_L,C_{I}),\sum_{L\neq I}d(\E_{L,I},\E_{I,L})\Big)$,
	\end{center}
where $L$ and $I$ range over all maximal spherical subsets of $S$.
For two Coxeter generating sets $S$ and $S_\tau$, we define $\K(S_\tau)<\K(S)$ if $\K_1(S_\tau)<\K_1(S)$, or $\K_1(S_\tau)=\K_1(S)$ and $\K_2(S_\tau)<\K_2(S)$.
\end{defn}

In the following lemma we prove that elementary twists preserve exposed $L$. This will enable us later to trace the change of $\K_2(S)$.

\begin{defn}
\label{def:correspondence_spherical}
Let $L\subset S$ be maximal spherical and let $\tau$ be an elementary twist with $S\setminus (J\cup J^\perp)=A\sqcup B$ as in the definition of an elementary twist in the Introduction. We define the following spherical subset $L_\tau\subset \tau(S)$. If $L\subseteq A\cup J\cup J^\perp$, then we set $L_\tau=L$. If $L\subseteq B\cup J\cup J^\perp$, then we set $L_\tau=w_JLw_J^{-1}$. Note that this definition is not ambiguous if $L\subseteq J\cup J^\perp$, since then by the maximality of $L$ we have $J\subseteq L$ and hence $L=w_JLw_J^{-1}$. If $L'$ is an irreducible
subset of $L$, then similarly $L'_\tau$ denotes $L'$ or $w_JL'w_J^{-1}$ depending on whether $L\subseteq A\cup J\cup J^\perp$ or $L\subseteq B\cup J\cup J^\perp$ as before. Note that $L'_\tau$ might depend on $L$, but only if $L'\subsetneq J$. In particular this cannot happen for $|L'|\geq 3$ and $|J|=2$.
\end{defn}

Note that $L_\tau\subset \tau(S)$ is still maximal spherical and the assignment $L\to L_\tau$ is a bijection between the maximal spherical subsets of $S$ and $\tau(S)$.

\begin{lem}
\label{lem:staygood}
Let $\tau$ be an elementary twist of $S$. Let $L$ be a maximal irreducible subset of a maximal spherical subset of $S$. If $|L|=3$ and $L$ is exposed in $S$, then $L_\tau$ is exposed in $\tau(S)$.
\end{lem}

\begin{proof}
We can assume that $\tau$ is an elementary twist with $J=\{s,t\}$ and $m_{st}$ odd, since otherwise the defining graph of $S$ is invariant under~$\tau$. Let $S\setminus (\{s,t\}\cup \{s,t\}^\perp)$ decompose into $A\sqcup B$ as in the definition of an elementary twist. Without loss of generality assume $L\subset A\cup \{s,t\}\cup \{s,t\}^\perp$. Then $L_\tau=L$.

First consider the case where $\{s,t\}$ is disjoint from $L$. Then $l\in L$ is adjacent to $r\in S$ if and only if $\tau(l)=l$ is adjacent to $\tau(r)$ in the defining graph of $\tau(S)$ and $m_{lr}=m_{\tau(l)\tau(r)}$. In particular $(L_\tau)^\perp=\tau({L}^\perp)$. Then $L_\tau$ is exposed.

Secondly, consider the case where $\{s,t\}\subset L$. Then $(L_\tau)^\perp=\tau({L}^\perp)$. Moreover, since $k=1$ or $2$ elements among $s,t$ are not adjacent to any element of $S\setminus (L\cup {L}^\perp)$, we have that $S\setminus (L\cup {L}^\perp)$ is contained entirely in $A$ or $B$. Consequently, there are $k$ elements among $\tau(s)=s,\tau(t)=t$ that are not adjacent to any elements of $\tau (S\setminus (L\cup {L}^\perp))$, and thus $L_\tau$ is exposed.

Thirdly, consider the case where $\{s,t\}\cap L=\{t\}$. Then $t$ is the only element of $L$ adjacent to some element of $S\setminus (L\cup {L}^\perp)$. Note also that since $t\in L, s\notin L$, we have ${L}^\perp\subseteq A\cup \{s,t\}^\perp$. Thus $(L_\tau)^\perp=\tau({L}^\perp)$. Moreover, none of the elements of $\tau (S\setminus (L\cup L^\perp))$ is adjacent to an element of $L_\tau\setminus \{t\}$, and so $L_\tau$ is exposed.
\end{proof}

\begin{remark}
\label{rem:L_isubsetI}
\begin{enumerate}[(i)]
\item
Suppose that $L_i$ in Definition~\ref{def:E} is not exposed and that we have $L_i\subset I$. Then we also have $(L_i)_\tau\subset I_\tau$ in $\tau(S)$. Thus by Lemma~\ref{lem:staygood} we have $\E^i_{L,I}=D_L$ if and only if $\E^i_{L_\tau,I_\tau}=D_{L_\tau}$.
\item
Consequently, by Remark~\ref{def:alternateE}, if $L\cup I\subseteq A\cup J\cup J^\perp$, then $\E_{(L_i)_\tau,I_\tau}=\E_{L_i,I}$ and so in particular we have $d(\E_{L,I},\E_{I,L})=d(\E_{L_\tau,I_\tau},\E_{I_\tau,L_\tau})$. We have the same conclusion for
$L\cup I\subseteq B\cup J\cup J^\perp$, since then $\E_{(L_i)_\tau,I_\tau}=w_J\E_{L_i,I}$.
\item
Suppose that $L\subseteq A\cup J\cup J^\perp$ and $I\subseteq B\cup J\cup J^\perp$ with $J\subseteq I$. Then $C_I=w_JC_I=C_{I_\tau}$. Consequently, by Remark~\ref{def:alternateE} we have $\E_{(L_i)_\tau,I_\tau}=\E_{L_i,I}$.
Analogously, if $L\subseteq B\cup J\cup J^\perp$ and $J\subseteq I$, then $\E_{(L_i)_\tau,I_\tau}=w_J\E_{L_i,I}$.
\end{enumerate}
\end{remark}

\section{Proof of the main theorem : exposed components}
\label{sec:proof}

We keep the setup from Section~\ref{sec:independent}.
The Main Theorem reduces to the following.

\begin{thm}
\label{thm:minimal complexity}
Let $S$ be a Coxeter generating set of type $\mathrm{FC}$ angle-compatible with a Coxeter generating set $S'$. Suppose that any Coxeter generating set twist-equivalent to $S$ is $3$-rigid. Assume moreover that $S$ has minimal complexity among all Coxeter generating sets twist-equivalent to $S$. Then $S$ is conjugate to $S'$.
\end{thm}

The main step in the proof of Theorem~\ref{thm:minimal complexity} will be to establish the consistency of doubles.

\begin{defn}
	\label{def:consistent doubles}
Let $S$ be a Coxeter generating set and let $I\subset S$ be irreducible spherical with $|I|=2$. We say that $I$ is \emph{consistent}
if for any simple markings $\mu_1,\mu_2$ with supports containing $I$ and cores $s_1,s_2\in I$ the pair $\Phi_{s_1}^{\mu_1}, \Phi_{s_2}^{\mu_2}$ is geometric (which means $\Phi_{s_1}^{\mu_1}=\Phi_{s_2}^{\mu_2}$ for $s_1=s_2$).
Otherwise we say that $I$ is \emph{inconsistent}.
We say that $S$ has \emph{consistent doubles}, if any such $I$ is consistent.
\end{defn}

In the following we use the notation from Definition~\ref{def:K}.

\begin{defn}
\label{def:compatible}
Let $\{s,t\}\subset S$ be irreducible spherical. We say that components $A_1,A_2$ of $S\setminus(\{s,t\}\cup \{s,t\}^{\perp})$ are \emph{compatible} if $\Phi_s^{t,A_1}=\Phi_s^{t,A_2}$ and $\Phi_t^{s,A_1}=\Phi_t^{s,A_2}$. We say that a component $A$ of $S\setminus(\{s,t\}\cup \{s,t\}^{\perp})$ is \emph{self-compatible} if the pair $\Phi^{t,A}_{s},\Phi^{s,A}_{t}$ is geometric.
\end{defn}
Note that if all components of $\{s,t\}$ are compatible and self-compatible, then $\{s,t\}$ is consistent. We will prove the compatibility in different ways depending on the type of the components.

\begin{defn}
\label{defn:big}
Let $\{s,t\}\subset S$ be irreducible spherical. A component~$A$ of $S\setminus (\{s,t\}\cup \{s,t\}^\perp)$ is \emph{big} if there is $r\in A$ with $\{s,t,r\}$ not spherical. Otherwise $A$ is \emph{small}.
We say that a component $A$ of $S\setminus (\{s,t\}\cup \{s,t\}^\perp)$ is \emph{exposed} if there is $p\in A$
such that $\{s,t,p\}$ is exposed (See Definition~\ref{def:goodJ}). An exposed component might be small or big. 
\end{defn}

The goal of this section is the following.

\begin{prop}
\label{lem:exposed}
Under the assumptions of Theorem~\ref{thm:minimal complexity}, if there is an exposed component of $S\setminus(\{s,t\}\cup\{s,t\}^\perp)$, then $\{s,t\}$ is consistent.
\end{prop}

In the proof we will need the following terminology and lemmas.

\begin{defn}
Let $J\subset S$ be irreducible spherical. By $\mathcal W_J$ we denote the union of $\mathcal W_j$ over all reflections $j\in \langle J\rangle$. The components of $\Da\setminus \mathcal W_J$ are called \emph{sectors for $J$}. The two sectors containing the geometric fundamental domains for $J$ are called \emph{geometric}.
\end{defn}

\begin{lem}
\label{lem:triple_not2good}
Under the hypotheses of Theorem~\ref{thm:minimal complexity},
let $J\subset S$ be exposed with $|J|=3$.
Suppose that we have simple markings $\mu_1,\mu_2$ with supports contained in $J$, and cores $s_1,s_2$.
Then the pair $\Phi_{s_1}^{\mu_1}, \Phi_{s_2}^{\mu_2}$ is geometric.
\end{lem}

\begin{proof}
\textbf{Case 1. The unique component of $S\setminus (J\cup J^\perp)$ has no element adjacent to an element of~$J$.}
Since $S\setminus (J\cup J^\perp)$ is a single component, all the walls $\mathcal W_{r}$ for $r\in S\setminus (J\cup J^\perp)$ lie in $\Da$ in a single sector $\Lambda$ for $J$.
If $\Lambda$ is a geometric sector, then the pair $\Phi_{s_1}^{\mu_1}, \Phi_{s_2}^{\mu_2}$ is geometric, since by Lemma~\ref{lem:spherical same side} each $\Phi_{s_i}^{\mu_i}$ is the halfspace for~$s_i$ containing $\Lambda$.

If $\Lambda$ is not geometric, suppose that it is of form $w\Lambda_0$ for $\Lambda_0$ a geometric sector for $J$ and $w\in \langle J \rangle$. Let $w=t_0\cdots t_{n-1}$ with $t_i\in J$ and minimal $n$. Consider the following Coxeter generating sets $S_i$ with $S_0=S$ and elementary twists $\tau_i$ with $S_{i+1}=\tau_i(S_i)$. Namely, we set $J_i=\{t_i\}, A_i=J\setminus (t_i\cup t_i^\perp)$ (which is a component of $S\setminus (t_i\cup t_i^\perp)$), and $B_i=S\setminus (J_i\cup J_i^\perp\cup A_i)$. The elementary twist $\tau_i$ conjugates $B_i$ by~$t_i$ and fixes the other elements of $S_i$. Let $\tau=\tau_{n-1}\circ \cdots \circ \tau_0$, so that $S_n=\tau(S)$.

We now argue, similarly as in \cite[\S7.2]{HP}, that $\mathcal K_1(\tau(S))=\mathcal K_1(S)$ and $\mathcal K_2(\tau(S))<\mathcal K_2(S)$. A maximal spherical subset $L$ of $S$ either contains~$J$, and is then called \emph{idle} or intersects $S\setminus (J\cup J^\perp)$. Thus all $D_I$ with $I\subset S$ maximal spherical that are not idle, are contained in $\Lambda$.
For $L$ idle we have $D_{L_\tau}=D_L$. In particular,
for all maximal spherical $I\subset S$ we have $C_{I_\tau}=w^{-1}C_I$, implying $\mathcal K_1(\tau(S))=\mathcal K_1(S)$.

To compare $\mathcal K_2(\tau(S))$ and $\mathcal K_2(S)$, first note that
if both $L$ and $I$ are maximal spherical and idle (resp.\ not idle), then by Remark~\ref{rem:L_isubsetI}(ii) we have $d(\E_{L,I},\E_{I,L})=d(\E_{L_\tau,I_\tau},\E_{I_\tau,L_\tau})$.
Now suppose that $L$ is idle and $I$ is not idle.
Then by Remark~\ref{rem:L_isubsetI}(iii) we have
$\E_{I_\tau,L_\tau}=w^{-1}\E_{I,L}$. Furthermore, $J\subset L$ is maximal irreducible, and so the decomposition of $L$ into maximal irreducible subsets has the form $L=L_1\sqcup \cdots \sqcup L_k$ with $L_1=J$. Since $J$ is exposed, by Lemma~\ref{lem:staygood} we have $\E^1_{L_\tau,I_\tau}= D(L_\tau)=D(L)=\E^1_{L,I}$. For $i\neq 1$ we have $L_i\subseteq J^\perp$ and so by Remark~\ref{rem:L_isubsetI}(ii) we have $\E^i_{L_\tau,I_\tau}=w^{-1}\E^i_{L,I}$, which equals $\E^i_{L,I}$ since $w$ commutes with $L_i$. Thus $\E_{L_\tau,I_\tau}=\E_{L,I}$. Let $\beta=\beta'\beta''$ be a minimal gallery from a chamber in $E_{I,L}$ to a chamber $x\in E_{L,I}$, where $\beta'\subset \Lambda$ and $\beta''$ is contained in the $J$-residue containing $x$. (Such a gallery exists by \cite[Thm 2.9]{R}.) Then $w^{-1}\beta'$ connects a chamber in $\E_{I_\tau,L_\tau}$ to a chamber in $\E_{L_\tau,I_\tau}$, proving $\mathcal K_2(\tau(S))<\mathcal K_2(S)$.

\smallskip \noindent
\textbf{Case 2. The unique component of $S\setminus (J\cup J^\perp)$ has an element $r'$ adjacent to an element $t\in J$.}
Let $\Lambda$ be a sector for $J$ with $\Lambda\cup t\Lambda$ containing $\mathcal W_{r'}$.
If $\Lambda$ or $t\Lambda$
is a geometric sector, then the pair $\Phi_{s_1}^{\mu_1}, \Phi_{s_2}^{\mu_2}$ is geometric as in Case~1.
Suppose now that neither
$\Lambda$ nor $t\Lambda$
is geometric. Let $w\in \langle J \rangle$ be of minimal word length with $w\Lambda_0=\Lambda$ or~$t\Lambda$ and
$\Lambda_0$ a geometric sector for $J$. Say we have $w\Lambda_0=\Lambda$.

Since $\mathcal W_{r'}$ intersects $\mathcal W_t$, there is $t'\in J$ satisfying $wt'w^{-1}=t$. By \cite[Prop~5.5]{Deod} there is $n\geq 0$, elements $t_0=t, \ldots, t_n=t'\in J$ and $s_0, \ldots s_{n-1}\in J$ such that for each $i=0,\ldots, n-1$ we have $s_{i}\neq t_{i}$, and for
$$
w_i=
\begin{cases}
s_i, \text{ if } s_i \text{ and } t_{i} \text{ commute }\\
\text{the longest word in } \langle s_{i},t_{i} \rangle, \text{ otherwise}\\
\end{cases}
$$
we have
\begin{itemize}
\item
$w_it_{i+1}w_i^{-1}=t_i$, and
\item
$w=w_0\cdots w_{n-1}$ or $w=tw_0\cdots w_{n-1}$.
\end{itemize}
We focus on the case where $w=w_0\cdots w_{n-1}$. Construct the following Coxeter generating sets $S_i\supset J$ with $S_0=S$ and elementary twists~$\tau_i$ with $S_{i+1}=\tau_i(S_{i})$.
We will also get inductively that the unique component of $S_i\setminus (J\cup J^\perp)$ does not have an element adjacent to an element of $J$ distinct from $t_i$.

If $s_i$ and $t_{i}$ commute, we set $J_i=\{s_i\}, A_i=J\setminus (s_i\cup s_i^\perp), B_i=S\setminus (J_i\cup J_i^\perp\cup A_i)$. The elementary twist $\tau_i$ conjugates $B_i$ by $s_i=w_i$ and fixes the other elements of $S_i$. Note that $A_i$ is a component of $S\setminus (J_i\cup J_i^\perp)$ since $t_i\in s_i^\perp$.
If $s_i$ and $t_{i}$ do not commute, we set $J_i=\{s_i,t_i\}$ and keep the same formulas for $A_i,B_i$. Then the elementary twist $\tau_i$ conjugates~$B_i$ by $w_i$ and fixes the other elements of $S_i$.

We argue analogously as in Case~1 to obtain $\mathcal K_1(S_n)=\mathcal K_1(S_0)$. For $L$ idle and $I$ not idle we also obtain analogously $\E_{I_\tau,L_\tau}=w^{-1}\E_{I,L},\E_{L_\tau,I_\tau}=\E_{L,I}$.
Let $\beta=\beta'\beta''$ be a minimal gallery from a chamber in $E_{I,L}$ to a chamber $x\in E_{L,I}$, where $\beta'\subset \Lambda$ or $t\Lambda$ and $\beta''$ is contained in the $J$-residue containing $x$. Then $w^{-1}\beta'$ connects a chamber in $\E_{I_\tau,L_\tau}$ to $x\in\E_{L_\tau,I_\tau}$ for $\beta'\subset \Lambda$ or to a chamber adjacent to $x$ for $\beta'\subset t\Lambda$. Moreover, in the latter case $\beta''$ has length at least $2$ by the minimality assumption on $w$.
This shows $d(\E_{L_\tau,I_\tau},\E_{I_\tau,L_\tau})<d(\E_{L,I},\E_{I,L})$ and hence $\mathcal K_2(S_n)<\mathcal K_2(S_0)$.

If $w=tw_0\cdots w_{n-1}$, then we start with an additional elementary twist in $\langle t\rangle$ and we continue analogously.
\end{proof}

\begin{lem}
\label{lem:uniquesmall_self}
Under the assumptions of Theorem~\ref{thm:minimal complexity}, if $S\setminus(\{s,t\}\cup\{s,t\}^\perp)$ is a single component that is small, then it is self-compatible.
\end{lem}

\begin{proof}
Let $A=S\setminus(\{s,t\}\cup\{s,t\}^\perp)$.
Let $\mu$ be a simple marking with support $J$ containing $s,t$ guaranteed by Remark~\ref{rem:find markings}. Let $q\in J$ be not adjacent to the marker $r$ of $\mu$ and suppose without loss of generality that $t$ lies on the minimal length path $\alpha$ from $s$ to $q$ in the Coxeter--Dynkin diagram of $J$. Then, replacing $J$ with the set of elements on~$\alpha$, we can assume that the Coxeter--Dynkin diagram of $J$ is a path starting with~$s$. Thus $\mu=((s,tpw),r))$ where $p\in A$ and $s$ commutes with $pw$. Since $A$ is the unique component of $S\setminus(\{s,t\}\cup\{s,t\}^\perp)$, we have that $r$ is adjacent to $s$. Thus $\mathcal W=pw\mathcal W_r$ intersects $\mathcal W_s$ and so by Corollary~\ref{cor:extra} there is a geometric fundamental domain $F$ for $\{s,t\}$ that is contained in both $\Phi(t\mathcal W_s,\mathcal W)$, and $\Phi(\mathcal W_t,\mathcal W)=\Phi(\mathcal W_t,s\mathcal W)$, which is $\Phi_t^{\mu'}$ for $\mu'=((t,spw),r)$. Thus by Lemma~\ref{lem:spherical same side} the pair $\Phi^\mu_s, \Phi_t^{\mu'}$ is geometric, as desired.
\end{proof}

\begin{proof}[Proof of Proposition~\ref{lem:exposed}]
Let $J=\{s,t,p\}$ be exposed and let $A$ be the component of $S\setminus (\{s,t\}\cup \{s,t\}^\perp)$ containing $p$. Consider first the case where $A=S\setminus (\{s,t\}\cup \{s,t\}^\perp)$. If $A$ is small, then it suffices to apply Lemma~\ref{lem:uniquesmall_self}. If $A$ is big, then let $r\in A$ with $\{s,t,r\}$ not spherical. By Lemma~\ref{lem:triple_not2good}, the halfspaces for $s,t$ determined by the markings $((s,t),r),((t,s),r)$ are geometric and hence $A$ is self-compatible.

It remains to consider the case where $A\subsetneq S\setminus (\{s,t\}\cup \{s,t\}^\perp)$.
Let $B$ be a component of $S\setminus (\{s,t\}\cup \{s,t\}^\perp)$ distinct from $A$ and let $r\in B$. Since the unique component of $S\setminus (J\cup J^\perp)$ has no element adjacent to one of $s,t$, we have that $\{s,t,r\}$ is not spherical. By Lemma~\ref{lem:triple_not2good}, the halfspaces for $s,t$ determined by the markings $((s,t),r), ((s,tp),r), ((t,s),r),((t,sp),r)$ are geometric and hence $B$ is compatible with $A$ and they are both self-compatible.
\end{proof}

\begin{cor}
\label{cor:small_self}
Under the assumptions of Theorem~\ref{thm:minimal complexity}, each small component of $S\setminus(\{s,t\}\cup\{s,t\}^\perp)$ is self-compatible.
\end{cor}
\begin{proof} Let $A$ be a small component of $S\setminus(\{s,t\}\cup\{s,t\}^\perp)$.
If $A=S\setminus(\{s,t\}\cup\{s,t\}^\perp)$, then it suffices to apply Lemma~\ref{lem:uniquesmall_self}.
Otherwise, let $r$ be an element of a component $B$ of $S\setminus(\{s,t\}\cup\{s,t\}^\perp)$ distinct from $A$.
Let $p\in A$, suppose without loss of generality $m_{sp}=2$, and set $\mu=((s,tp),r)$.
If $A$ is exposed, then we can apply Proposition~\ref{lem:exposed}. Otherwise, by Lemma~\ref{lem:exposed_criteria2} we have that $\{t,p\}\subset \{s,t,p\}$ is good with respect to $r$. Thus by the $3$-rigidity of $S$ and Proposition~\ref{prop:compactible1}, we have $\Phi(\mathcal W_t,p\mathcal W_r)=\Phi(\mathcal W_t,ps\mathcal W_r)$. By Lemma~\ref{sublem:extra} there is a geometric fundamental domain $F$ for $\{s,t\}$ that is contained in both $\Phi(t\mathcal W_s,p\mathcal W_r)$ and $\Phi(\mathcal W_t,p\mathcal W_r)$, which is $\Phi_t^{\mu'}$ for $\mu'=((t,sp),r)$. Thus by Lemma~\ref{lem:spherical same side} the pair $\Phi^\mu_s, \Phi_t^{\mu'}$ is geometric, and so $A$ is self-compatible.
\end{proof}

\section{Big components}
\label{sec:big}

The content of this section was designed together with Pierre-Emmanuel Caprace.

\begin{lem}
\label{lem:mst=3}
Under the assumptions of Theorem~\ref{thm:minimal complexity}, if $m_{st}=3$, then $\{s,t\}$ is consistent.
\end{lem}

\begin{proof}
By Proposition~\ref{lem:exposed} we can assume that no component of $S\setminus (\{s,t\}\cup \{s,t\}^\perp)$ is exposed.

For a big component $B$ of $S\setminus (\{s,t\}\cup \{s,t\}^\perp)$ and $r\in B$ with $\{s,t,r\}$ not spherical, let $F$ be the geometric fundamental domain for $\{s,t\}$ lying in $\Phi(s\mathcal W_t, \mathcal W_r)= \Phi(t\mathcal W_s, \mathcal W_r)$. By Lemma~\ref{lem:spherical same side} we have that $F$ lies in
$\Phi_s^{t,B}\cap\Phi_t^{s,B}$, which thus form a geometric pair. Hence $B$ is self-compatible and by Corollary~\ref{cor:small_self} it remains to prove that all $\Phi_s^{t,B}$ coincide (including small $B$).

Otherwise, let $A$ be the union of all components $A_i$ with one $\Phi_s^{t,A_i}$, which we call $\Phi_s^{t,A}$, and let $B$ be the union of components $B_i$ with the other $\Phi_s^{t,B_i}$, which we call $\Phi_s^{t,B}$. Let $\tau$ be the elementary twist that sends each element $b\in B$ to $w_{st}bw^{-1}_{st}$, where $w_{st}=tst$, and fixes the other elements of $S$. For a contradiction, we will first prove that if there are incompatible big components, then $\mathcal K_1(\tau(S))< \mathcal K_1(S)$. For maximal spherical $L\subset S$ we say that $L$ is \emph{twisted} if it contains an element of $B$.
We then have $C_{L_\tau}=w_{st}C_L$. If $I$ is maximal spherical and not twisted, then we have $C_{I_\tau}=C_I$.
Consequently $d(C_{L_\tau}, C_{I_\tau})$ might vary from $d(C_L, C_I)$ only if, say, $L$ is twisted and $I$ is not twisted, and $\{s,t\}\not\subseteq L,I$. Such $L,I$ exist exactly if there are incompatible big components. Then $C_L, C_I$ lie in the opposite halfspaces of $t\mathcal W_s=\mathcal W_{w_{st}}$, and consequently $d(C_{L_\tau}, C_{I_\tau})<d(C_L, C_I)$, as desired.

If all big components are compatible, we have $\mathcal K_1(\tau(S))= \mathcal K_1(S)$, and we need to analyse the effect of $\tau$ on $\mathcal K_2$. 

\smallskip

\textbf{Claim. } \emph{Let $L,I\subset S$ be maximal spherical subsets with $\{s,t\}\subset L$. Suppose that $I$ contains an element of $A$ and $L$ contains an element of~$B$ (or vice versa). Then we have $\E_{L,I}\subset \Phi_s^{t,B}$ (resp.\ $\E_{L,I}\subset \Phi_s^{t,A}$).}

\begin{proof}
Indeed, let $L_1\subseteq L$ be maximal irreducible containing $\{s,t\}$, and let $u\in L_1$ with $\{s,t,u\}$ irreducible, so that $u\in B$. Let $r\in I\setminus (\{s,t\}\cup \{s,t\}^\perp)$. Since the components of $S\setminus (\{s,t\}\cup \{s,t\}^\perp)$ are self-compatible, after possibly interchanging $s$ with $t$, we can assume that $u,s,t$ are consecutive in the Coxeter--Dynkin diagram of $L_1$. Then $s$ is not a leaf in the Coxeter--Dynkin diagram of $L_1$ and by Remark~\ref{def:alternateE} we have $\E_{L,I}\subset \Phi^\mu_s$ for $\mu=((s,L_1),r)$. Since $K^\mu_{s,t}\subset B$ (see Definition~\ref{def:K}), the claim follows. 
\end{proof}

Returning to the proof of the lemma, without loss of generality we can assume that all big components are contained in $A$. Consider maximal spherical subsets $L,I\subset S$.
If both $L,I$ are twisted, or both are not twisted, then by Remark~\ref{rem:L_isubsetI}(ii) we have
$d(\E_{L_\tau,I_\tau},\E_{I_\tau,L_\tau})=d(\E_{L,I},\E_{I,L})$.
Suppose now that $L$ is twisted and intersects a component $B_i\subseteq B$ and $I$ is not twisted. If $I\subseteq \{s,t\}\cup \{s,t\}^\perp$, the same equality holds, so we can assume $I\not\subseteq \{s,t\}\cup \{s,t\}^\perp$. Since $B_i$ is small, we have $\{s,t\}\subset L$. Consequently, by the claim we have $\E_{L,I}\subset \Phi_s^{t,B}$. The proof of the lemma splits now into two cases.

\textbf{Case 1. $I$ contains $\{s,t\}$.}
Interchanging the roles of $L$ and $I$, from the claim we have $\E_{I,L}\subset \Phi^{t,A}_{s}$. Consequently, $\E_{L,I}$ and $\E_{I,L}$ lie in the opposite geometric fundamental domains for $\{s,t\}$. In particular, they lie in the opposite halfspaces of $t\mathcal W_s=\mathcal W_{w_{st}}$. By Remark~\ref{rem:L_isubsetI}(iii), we have $\E_{L_\tau,I_\tau}=w_{st}\E_{L,I}$ and $\E_{I_\tau,L_\tau}=\E_{I,L}$. Thus $d(\E_{L_\tau,I_\tau}, \E_{I_\tau,L_\tau})<d(\E_{L,I}, \E_{I,L})$.

\textbf{Case 2. $I$ does not contain $\{s,t\}$.}
By FC, we have that $I$ contains an element $r$ not adjacent to $s$ or not adjacent to $t$.
By the claim and Lemma~\ref{lem:spherical same side}, we have $\E_{L,I}\subset t\Phi_s^{t,B}$.
Consider the marking $\mu=((s,t),r)$. Since $K^\mu_{s,t}\subseteq A$, we have $\mathcal W_r\subset t\Phi^{t,A}_{s}$, and so $\E_{I,L}\subset t\Phi^{t,A}_{s}$. Furthermore, we have $\E_{I_\tau,L_\tau}=\E_{I,L}$ as in Case~1. To finish as in Case~1, it remains to prove $\E_{L_\tau,I_\tau}=w_{st}\E_{L,I}$.

To this end, let $u\in L_1$ as in the proof of the claim. Note that in the Coxeter--Dynkin diagram of $(L_1)_\tau$ we have consecutive vertices $s,t$ and $\tau(u)$. We have that $(L_1)_\tau$ is not exposed by Lemma~\ref{lem:staygood}. By Lemma~\ref{lem:exposed_criteria2}, $u\in L_1, \{u,s\}\subset L_1$ are good with respect to $r$ and
$\tau(u)\in (L_1)_\tau, \{\tau(u),t\}\subset (L_1)_\tau$ are good with respect to $\tau(r)=r$. Consequently it suffices to prove $\Delta^{(\tau(u),t),r}=w_{st}\Delta^{(u,s),r}$. This follows from the fact that the reflections $sts$ and $sus$ commute, hence each of the halfspaces of $s\mathcal W_u$ is preserved by $w_{st}$, and thus $\Phi(tw_{st}\mathcal W_u,\mathcal W_r)=\Phi(w_{st}s\mathcal W_u,\mathcal W_r)=w_{st}\Phi(s\mathcal W_u,\mathcal W_r)$.
\end{proof}

For $m_{st}\neq 3$, we will need the following measure of consistency.

\begin{defn}
\label{def:peripheral}
Let $\{s,t\}\subset S$ be irreducible spherical and let $V$ be one of the two geometric fundamental domains for $\{s,t\}$. We define the \emph{consistency} $\mathcal C_V(s,t)=\mathcal C_V(t,s)$ as the number of maximal spherical $L\subset S$ with $C_L$ intersecting $sV\cup V\cup tV$. We say that inconsistent $\{s,t\}$ is \emph{peripheral} if $\mathcal C_V(s,t)$ is maximal among all inconsistent $\{s,t\}\subset S$ and both $V$.
\end{defn}

Obviously, if $S$ does not have consistent doubles, then there is peripheral $\{s,t\}$. The following remark describes the role of the union $sV\cup V\cup tV$.

\begin{remark}
\label{rem:self-comp} Let $\{s,t\}\subset S$ be irreducible spherical and let $V$ be a geometric fundamental domain for $\{s,t\}$. Suppose that we have $r\in S$ with $\mathcal W_r\subset sV\cup V\cup tV$. Then $\mu=((s,t),r),\mu'=((t,s),r)$ are markings, and we have $V\subset \Phi_s^\mu, \Phi_t^{\mu'}$. Consequently, the component of $S\setminus(\{s,t\}\cup\{s,t\}^\perp)$ containing $r$ is self-compatible. Conversely, if $\{s,t,r\}$ is not spherical, and the component $B$ of $S\setminus(\{s,t\}\cup\{s,t\}^\perp)$ containing $r$ is self-compatible, then $\mathcal W_r\subset sV\cup V\cup tV$ for a geometric fundamental domain $V$ for $\{s,t\}$. Furthermore, $V$ depends only on $B$, not on $r$.
\end{remark}

\begin{prop}
\label{prop:bigcompatible}
Under the assumptions of Theorem~\ref{thm:minimal complexity}, if $\{s,t\}$ is peripheral, then
big components of $S\setminus(\{s,t\}\cup\{s,t\}^\perp)$ are compatible. Moreover, if there is a big component that is not self-compatible, then all $\mathcal W_r$ with $\{s,t,r\}$ not spherical are contained in a single sector for~$\{ s,t\}$.
\end{prop}

In the proof we will need the following key notion.

\begin{defn}
\label{def:folding}
Let $\{s,t\}\subset S$ be irreducible spherical. A \emph{folding} is a map $f\colon \langle s,t\rangle\to \{s,\mathrm{Id},t\}$ such that for each $w\in \langle s,t\rangle$ we have
\begin{itemize}
\item
$f(ws)=f(w)$ or $f(ws)=f(w)s$, and
\item
$f(wt)=f(w)$ or $f(wt)=f(w)t$.
\end{itemize}
In other words, $f$ is a simplicial type-preserving map on the Cayley graph of $\langle s,t\rangle$.
\end{defn}

\begin{exa}
Let $m_{st}=3$ and $w_{st}=tst$. Let $f\colon \langle s,t\rangle\to \{s,\mathrm{Id},t\}$ be the map whose restriction to $\{s,\mathrm{Id},t\}$ is the identity map and whose restriction to $\{w_{st}s,w_{st},w_{st}t\}$ is the reflection $w_{st}$. It is easy to see that $f$ is a folding.
\end{exa}

\begin{lem}
\label{lem:folding}
Let $f\colon \langle s,t\rangle\to \{s,\mathrm{Id},t\}$ be a folding. Let $V$ be a geometric fundamental domain for $\{s,t\}$. Let $\tilde f\colon \Da^{(0)}\to sV\cup V\cup tV$ be the map sending each $wV$ to $f(w)V$ via $f(w)w^{-1}$, where $w\in \langle s, t\rangle$. Then $\tilde f$ induces a simplicial map on $\Da^{(1)}$. Moreover, for $x\in wV, y\in w'V$ we have $d(\tilde f(x),\tilde f(y))=d(x,y)$ if and only if the restriction of $f$ to the vertices of some path $\pi$ from $w$ to $w'$ in the Cayley graph of $\langle s,t\rangle$ is injective.
\end{lem}

\begin{proof}
To prove the first assertion, consider adjacent chambers $g,gp$ of  $\Da^{(1)}$, where $g\in W, p\in S'$. If $g,gp$ belong to the same $wV$, then $\tilde f (g)=f(w)w^{-1}g$ and $\tilde f(gp)=f(w)w^{-1}gp$ are obviously adjacent. If $g,gp$ belong to distinct translates of $V$, then
we have, say, $g\in wV, gp\in wsV$.  In that case we also have $wsw^{-1}g=gp$ and so $g$ and $gp$ are in the same orbit of the action of $\langle s, t\rangle$ on $\Da^{(0)}$. Consequently, if $f(ws)=f(w)$, then since $f(w)V$ intersects each $\langle s, t\rangle$-orbit in one chamber, we have $\tilde f(g)=\tilde f(gp)$. On the other hand, if $f(ws)=f(w)s$, then $f(ws)(ws)^{-1}=f(w)w^{-1}$, and hence $\tilde f(g)=f(w)w^{-1}g$ and $\tilde f(gp)=f(ws)(ws)^{-1}gp$ are adjacent.

For the second assertion, let $\gamma$ be a minimal gallery from $x$ to $y$, let $wV,\ldots, w'V$ be the distinct consecutive translates of $V$ traversed by $\gamma$ and let $\pi=w\cdots w'$ be the corresponding path in the Cayley graph of $\langle s,t\rangle$. If $d(\tilde f(x),\tilde f(y))=d(x,y)$, then in view of the previous paragraph the consecutive vertices of the path $f(\pi)$ are distinct, as desired. Conversely, if $d(\tilde f(x),\tilde f(y))<d(x,y)$, then a pair of consecutive vertices of $f(\pi)$ coincides. Since $\gamma$ was minimal, the length of of $\pi$ is at most $m_{st}$, and consequently the length of the second path $\pi'$ from $w$ to $w'$ in the Cayley graph of $\langle s,t\rangle$ is $\geq m_{st}>2$. Since $f$ takes only values $s,\mathrm{Id},t$, the restriction of $f$ to $\pi'$ is also not injective.
\end{proof}

\begin{proof}[Proof of Proposition~\ref{prop:bigcompatible}]
Let $\Lambda_0$ be the geometric sector containing~$V$ from Definition~\ref{def:peripheral}. First consider the case where $m_{st}$ is odd, so the longest word $w_{st}$ in $\langle s,t\rangle$ is a reflection. This case will not require the peripherality hypothesis.

We begin with focusing entirely on the case where a component $B$ of $S\setminus (\{s,t\}\cup \{s,t\}^\perp)$ is not self-compatible. Observe that if $p\in B$ is adjacent to $s$, then it is also adjacent to $t$ (and vice versa): indeed, otherwise the pair of halfspaces determined by markings $((s,t),p), ((t,s),p)$ would be geometric by Corollary~\ref{cor:extra} (and Lemma~\ref{lem:spherical same side}).

We now claim that if $r\in B$ is not adjacent to $s$, then $s\in \{s,t\}$ is not good with respect to $r$. Indeed, otherwise let $\mu_1=((s,\emptyset), r), \mu_2=((s,t),r)$. If $\Phi_s^{\mu_1}=\Phi_s^{\mu_2}$, then Lemma~\ref{sublem:extra} (and Lemma~\ref{lem:spherical same side}) contradict the assumption that $B$ is not self-compatible. Thus by Proposition~\ref{prop:same side2}, there is a vertex $p\neq t$ on a minimal length path from $r$ to~$t$ in the defining graph of $S$ outside $s\cup s^\perp$, with $p$ adjacent to $s$ and $\{s,p\}$ inconsistent. By Lemma~\ref{lem:mst=3}, we have $m_{st},m_{sp}>3,$ so from $\mathrm{FC}$ it follows that $p$ is not adjacent to $t$. This contradicts the observation above, and justifies the claim.

Analogously $t\in \{s,t\}$ is not good with respect to $r$. Consequently, by Lemma~\ref{lem:stnotgood}, the elements of $B$ are adjacent neither to $s$ nor $t$. Thus $B$ is also a component of $S\setminus (s\cup s^\perp)$ and a component of $S\setminus (t\cup t^\perp)$. Furthermore, all $\mathcal W_r$ for $r\in B$ lie in a single sector $w_B\Lambda_0$ for some $w_B\in \langle s, t\rangle$ and by Remark~\ref{rem:self-comp} we have $w_B\neq s, \mathrm{Id},t,w_{st}s, w_{st},w_{st}t$. Consequently, for $L$ maximal spherical intersecting $B$, we have $C_L\subset w_BV$.

Let $j$ be the first letter in the minimal length word representing $w_B$. We set $\tau_B$ to be the composition of elementary twists conjugating $B$ by the letter $s$ or $t$ in the order in which they appear as consecutive letters in $w_Bj$. As a result, for $L$ maximal spherical intersecting $B$, we have $C_{L_{\tau_B}}=jw^{-1}_BC_L\subset jV$. (Here by $L_{\tau_B}$ with $\tau_B$ a composition $\tau_n\circ \cdots \circ \tau_1$ of elementary twists and $\sigma=\tau_{n-1}\circ \cdots \circ\tau_1$ we mean, inductively, $(L_\sigma)_{\tau_n}$.)

Consider now a self-compatible component $B$ of $S\setminus (\{s,t\}\cup \{s,t\}^\perp)$. By Remark~\ref{rem:self-comp} either
\begin{enumerate}[(i)]
\item
for each $L$ maximal spherical intersecting $B$, we have that $C_L$ intersects $sV\cup V\cup tV$, or
\item
each such $C_L$ intersects $w_{st}sV\cup w_{st}V\cup w_{st}tV$.
\end{enumerate}
If $B$ is big, then there is $L$ for which we can replace the word `intersects' by `is contained in' in the preceding statement. In case (ii),
we perform an elementary twist $\tau_B$ with $J=\{s,t\}$, which sends each $p\in B$ to $w_{st}pw_{st}^{-1}$. As a result, for $L$ maximal spherical intersecting~$B$, we have $C_{L_{\tau_B}}=w_{st}C_L$, which intersects $sV\cup V\cup tV$. Let $\tau$ be the composition of all $\tau_B$ above.

To summarise, consider the folding $f\colon \langle s,t\rangle\to \{s,\mathrm{Id},t\}$ defined by:
\begin{itemize}
\item $f(w)=w$ for $w=s,\mathrm{Id},t$,
\item $f(w)=w_{st}w$ for $w=w_{st}s,w_{st},w_{st}t$,
\item $f(w)=j$ for other $w$, where $j$ is the first letter in the minimal length word representing $w$.
\end{itemize}
By Lemma~\ref{lem:folding}, for $x\in wV, y\in w'V$ we have $d(\tilde f(x),\tilde f(y))\leq d(x,y)$ with equality if and only if $w=w'$ or both $w,w'$ lie in $\{s,\mathrm{Id},t\}$ or they both lie in $\{w_{st}s,w_{st},w_{st}t\}$.
Furthermore, for each $L$ maximal spherical we have $C_{L_\tau}\supseteq\tilde{f}(C_L)$
(where the inclusion is strict exactly when $L\supseteq \{s,t\}$). Thus we get
$\mathcal K_1(\tau(S))\leq \mathcal K_1(S)$. Moreover, we have strict inequality as soon as there are two incompatible big components or a big component $B$ that is not self-compatible, and $\mathcal W_r\not\subset w_B\Lambda_0$ with $\{s,t,r\}$ not spherical.

Secondly, consider the case where $m_{st}$ is even. We treat components~$B$ that are not self-compatible exactly as before. Suppose now that $B$ is a self-compatible component of $S\setminus (\{s,t\}\cup \{s,t\}^\perp)$ as in case~(ii).
A \emph{refined component} of $B$ is a component of $B\setminus (s^\perp\cup t^\perp)$.

Let $L\subset S$ be maximal spherical intersecting $B$. Suppose that $C_L$ does not intersect $w_{st}V$. Then it is contained in one of $sw_{st}V, tw_{st}V$, say $sw_{st}V$. By the maximality of $L$, there is $r\in L$ that is not adjacent to $s$ and so $\mathcal W_r\subset sw_{st}\Lambda_0$. In particular $r$ is not adjacent to $t$ and so $r$ lies in a refined component $B'$ of $B$.
We claim that $s\in \{s,t\}$ is not good with respect to $r$. Indeed, otherwise as before let $\mu_1=((s,\emptyset), r), \mu_2=((s,t),r)$ so that $\Phi_s^{\mu_1}\neq \Phi_s^{\mu_2}$. Thus by Proposition~\ref{prop:same side2}, there is a vertex $p\neq t$ on a minimal length path $\omega$ from $r$ to $t$ in the defining graph of $S$ outside $s\cup s^\perp$, with $p$ adjacent to $s$ and $\{s,p\}$ inconsistent. Note that all the vertices of $\omega$ distinct from $t$ lie in $B'$ except for possibly the vertex preceding $t$ that might lie in $B\cap t^\perp$, which is excluded below.

By Lemma~\ref{lem:mst=3}, we have again $m_{st},m_{sp}>3,$ so from $\mathrm{FC}$ it follows that $p$ is not adjacent to $t$. Then $p\mathcal W_s, s\mathcal W_p$ are disjoint from $s\mathcal W_t, t\mathcal W_s$. Moreover, since $p\in B'$, we have $\mathcal W_p\subset w_{st}s\Lambda_0\cup w_{st}\Lambda_0\cup w_{st}t\Lambda_0$, and so $p\mathcal W_s, s\mathcal W_p\subset w_{st}s\Lambda_0\cup w_{st}\Lambda_0\cup w_{st}t\Lambda_0$. Consequently, there is a geometric fundamental domain $V'$ for $\{ s,p \}$ such that $sV'\cup V'\cup pV'$ contains~$sV\cup V\cup tV$ and intersects $C_L$ for some $L$ maximal spherical containing $s,p$. Since $C_L$ is disjoint from $sV\cup V\cup tV$, we have $\mathcal C_{V'}(s,p)>\mathcal C_V(s,t)$, contradicting the hypothesis that $\{s,t\}$ is peripheral. This justifies the claim.

By the claim, there is no element in $B'$ adjacent to $t$ or to $B\cap t^\perp$. Thus $B'$ is a component of $S\setminus (s\cup s^\perp)$.

Furthermore, we will prove that for each $L'\subset S$ maximal spherical intersecting $B'$ we have that $C_{L'}$ intersects $sw_{st}V$. Otherwise, for $r'\in L'$ that is not adjacent to $s$ we have $r'\in B'$ and
$\mathcal W_{r'}\subset w_{st}\Lambda_0$. Consequently, for $\mu_1=((s,\emptyset), r), \mu_2=((s,\emptyset),r')$ we have $\Phi_s^{\mu_1}\neq \Phi_s^{\mu_2}$. Then by Proposition~\ref{prop:same side2}, there is an element $p\in B'$ with $p$ adjacent to $s$ and $\{s,p\}$ inconsistent. As before, this contradicts the hypothesis that $\{s,t\}$ is peripheral.

Consequently, for each self-compatible component $B$ of $S\setminus (\{s,t\}\cup \{s,t\}^\perp)$ as in case (ii), and its refined component $B'$, there is at least one element $w_{B'}$ among $sw_{st},w_{st},tw_{st}$ with $C_L\subset w_{B'}V$ for all $L$ maximal spherical intersecting $B'$.

We perform now a sequence of elementary twists as follows. We treat the components that are not self-compatible as before. For a self-compatible component $B$ of $S\setminus (\{s,t\}\cup \{s,t\}^\perp)$ as in case (ii) we do the following. First, for each refined component $B'\subset B$ satisfying $w_{B'}=sw_{st}$ (resp.\ $w_{B'}=tw_{st}$) we apply the elementary twist with $J=\{s\}$ (resp.\ $J=\{t\}$) that conjugates all the elements of $B'$ by $s$ (resp.\ $t$) and fixes all the other elements of $S$. Afterwards, we apply the elementary twist with $J=\{s,t\}$ that conjugates the entire image of $B$ under the preceding elementary twists by $w_{st}$. Let $\tau$ be the composition of all these elementary twists. Then $C_{L_\tau}\supseteq\tilde{f}(C_L)$ with $f\colon \langle s,t\rangle\to \{s,\mathrm{Id},t\}$ the folding defined by:
\begin{itemize}
\item $f(w)=w$ for $w=s,\mathrm{Id},t$,
\item $f(w)=\mathrm{Id}$ for $w=w_{st}s,w_{st},w_{st}t$,
\item $f(w)=j$ for other $w$, where $j$ is the first letter in the minimal length word representing $w$.
\end{itemize}
We can thus apply Lemma~\ref{lem:folding} as before.
\end{proof}

\section{Small components}
\label{sec:small}

\begin{prop}
\label{prop:smallcompatible}
Under the assumptions of Theorem~\ref{thm:minimal complexity}, doubles are consistent.
\end{prop}

\begin{proof}
Otherwise, let $\{s,t\}\subset S$ be peripheral and let $V$ be as in Definition~\ref{def:peripheral}. By Proposition~\ref{lem:exposed}, Corollary~\ref{cor:small_self}, and Proposition~\ref{prop:bigcompatible}, we can assume that none of the components of $S\setminus(\{s,t\}\cup\{s,t\}^\perp)$ are exposed, that all small components are self-compatible, and that big components are compatible.
Thus it remains to prove that each small component is compatible with any other component and that all big components are self-compatible.
Divide the components into two families $\{A_i\}$ and~$\{B_i\}$ such that all $\Phi_s^{t,A_i}$ coincide and are distinct from all~$\Phi_s^{t,B_i}$, which also coincide. Let $A$ (resp.\ $B$) be the union of all $A_i$ (resp.\ $B_i$) and suppose that all big components are in $B$. Denote $\Phi_s^{t,A}=\Phi_s^{t,A_i}, \Phi_s^{t,B}=\Phi_s^{t,B_i}$.
If there are self-compatible big components, then this implies $V\subset\Phi_s^{t,B}$. If there are no self-compatible big components, then, after possibly switching~$V$, we can also assume $V\subset\Phi_s^{t,B}$.

Let $w_{st}$ be the longest word in $\langle s,t\rangle$ and for each $A_i$ let $\tau_{A_i}$ be the elementary twist that sends each element $a\in A_i$ to $w_{st}aw_{st}^{-1}$, and fixes the other elements of $S$. For a big component $B_i$ that is not self-compatible, we define $w_{B_i},\tau_{B_i}$ as in the fourth and fifth paragraph of the proof of Proposition~\ref{prop:bigcompatible}. Let $\tau$ be the composition of all these $\tau_{A_i}$ and $\tau_{B_i}$. Let $L\subset S$ be a maximal spherical subset. $L$ is \emph{twisted} if it contains an element of $A$. In that case $s,t\in L$. $L$ is \emph{rotated} if it contains an element of $B_i$ that is not self-compatible. If $L$ is neither twisted nor rotated, we call it \emph{idle}.

Note that if we have rotated subsets, then by Proposition~\ref{prop:bigcompatible} we have no idle subsets not containing $\{s,t\}$, and that all $w_{B_i}$ coincide. Consequently $\mathcal K_1(\tau(S))=\mathcal K_1(S)$. We will now prove $\mathcal K_2(\tau(S))< \mathcal K_2(S)$.

Consider maximal spherical subsets $L,I\subset S$.
If both $L,I$ are twisted, both are rotated, or both are idle, by Remark~\ref{rem:L_isubsetI}(ii) we have
$d(\E_{L_\tau,I_\tau},\E_{I_\tau,L_\tau})=d(\E_{L,I},\E_{I,L})$.
Suppose for a moment that $L$ is twisted and $I$ is rotated or idle. If $I\subseteq \{s,t\}\cup \{s,t\}^\perp$, the same equality holds, so we can assume $I\not\subseteq \{s,t\}\cup \{s,t\}^\perp$.
We then have $\E_{L,I}\subset \Phi^{t,A}_{s}$, word for word as in the proof of the claim in Lemma~\ref{lem:mst=3}, and so $\E_{L,I}\subset w_{st}V$. Analogously, if $L$ is idle and contains $\{s,t\}$, and $I$ is rotated or twisted, we have $\E_{L,I}\subset V$, except in the `special' case where $L\subseteq\{s,t\}\cup \{s,t\}^\perp$ and so, say, $L_1=\{s,t\}$ is exposed and $\E^1_{L,I}=D_L$. Furthermore, for $L$ idle not containing $\{s,t\}$ we have $C_L\subset sV\cup V\cup tV$, and for $L\subset B_i$ rotated we have $C_L\subset w_{B_i}V$. This accounts for all possible positions of $\E_{L,I}$. We now need to analyse the effect of $\tau$ on all $\E_{L,I}$. Let $f$ be the folding from the proof of Proposition~\ref{prop:bigcompatible}. We will prove that except in the `special' case where $L_1=\{s,t\}$, we have
\begin{equation}
\E_{L_\tau,I_\tau}=\tilde f (\E_{L,I}).\tag{$*$}
\end{equation}

\textbf{Case 1. $I$ is twisted or idle containing $\{s,t\}$.} Then ($*$) follows from Remark~\ref{rem:L_isubsetI}.(iii).

\textbf{Case 2. $I$ is rotated or idle not containing $\{s,t\}$.} In that case $L$ contains $\{s,t\}$. We choose $u\in L_1$ as in the proof of the claim in Lemma~\ref{lem:mst=3} (possibly interchanging $s$ with $t$). Suppose first that $I$ is idle not containing $\{s,t\}$, and so $L$ is twisted. Then ($*$) amounts to $\E_{L_\tau,I_\tau}=w_{st}\E_{L,I}$. Let $r\in I\setminus (\{s,t\}\cup \{s,t\}^\perp)$. We have $m_{st}=4$ or $5$. If $m_{st}=5$,  then each of the halfspaces of $ts\mathcal W_u$ is preserved by $w_{st}$, since the reflections $tstst$ and $tsust$ commute, and so $w_{st}\Phi(ts\mathcal W_u,\mathcal W_r)=\Phi(w_{st}ts\mathcal W_u,\mathcal W_r)$. This implies $\E_{L_\tau,I_\tau}=w_{st}\E_{L,I}$ as in Case~2 of the proof of Lemma~\ref{lem:mst=3}. If $m_{st}=4$, we have $(stst)ts\mathcal W_u=st\mathcal W_u=s\mathcal W_u$. Thus $w_{st}$ exchanges the halfspaces of $ts\mathcal W_u$ and $s\mathcal W_u$, and in fact acts on them as $t$ does, so in particular $\Phi(w_{st}s\mathcal W_u,\mathcal W_r)=\Phi(ts\mathcal W_u,\mathcal W_r)$. Since $\{u,s\}\subset L_1$ is good with respect to $r$, and $S$ is $3$-rigid, by Proposition~\ref{prop:compactible1} we have $\Phi(ts\mathcal W_u,\mathcal W_r)=w_{st}\Phi(s\mathcal W_u,\mathcal W_r)$, and ($*$) follows.

It remains to consider the case where $I$ is rotated. Let $r\in I$ and suppose first $m_{st}=4$. Let $K=\Phi(s\mathcal W_u,\mathcal W_r)\cap t\Phi(s\mathcal W_u,\mathcal W_r)$, which contains $\mathcal W_r$ as in the preceding paragraph. Since the pair $s\Phi(s\mathcal W_t, \mathcal W_r),t\Phi(t\mathcal W_s, \mathcal W_r)$ is not geometric, $\mathcal W_r$ may lie only in two sectors for $\{u,s,t\}$, indicated in Figure~2, left. Denoting by $\Sigma$ the union of the interiors of these two sectors, we have that $\langle s,t\rangle\Sigma$ lies entirely in $K$. This implies ($*$) for $L$ idle since $\Sigma$ and its image $\tilde f(\Sigma)$ under the folding lie in the same halfspace of $s\mathcal W_u$. It also implies ($*$) for $L$ twisted, since $\Phi(ts\mathcal W_u,\tilde f(\Sigma))=w_{st}\Phi(s\mathcal W_u,\Sigma)$. The case $m_{st}=5$ is similar: though the union $\Sigma$ of possible sectors containing $\mathcal W_r$ is larger (see Figure~2, right), its image $\tilde f(\Sigma)$ under the folding, in both possible cases for $V$, still lies entirely in one halfspace of $ts\mathcal W_u$, which is $w_{st}$ invariant.

\begin{figure}
\begin{center}
\includegraphics[width=0.9\textwidth]{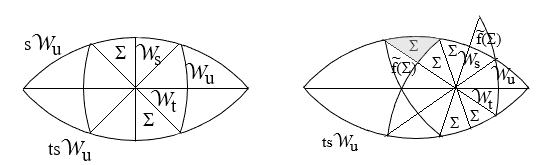}
\end{center}
\caption{On the right both possible positions of $\tilde f(\Sigma)$ for the shaded sector of $\Sigma$}
\end{figure}

This ends the proof of ($*$) as long as $L_1\neq\{s,t\}$. Then by Lemma~\ref{lem:folding}, as long as $L_1,I_1\neq\{s,t\},$ we have $d(\E_{L_\tau,I_\tau},\E_{I_\tau,L_\tau})\leq d(E_{L,I},E_{I,L})$, with strict inequality if $L,I$ are not both idle, both twisted or both rotated. It remains to consider the case where $I$ is idle with $I_1=\{s,t\}$. Recall that then $E^1_{I,L}=D_{I}=E^1_{I_\tau,L_\tau}$ and by Remark~\ref{rem:L_isubsetI}(ii) we have $E^i_{I,L}=E^i_{I_\tau,L_\tau}$ for $i>1$, and so in particular $E_{I_\tau,L_\tau}=w_{st}E_{I,L}$. Consequently, if~$L$ is twisted, by ($*$) we have $d(\E_{L_\tau,I_\tau},\E_{I_\tau,L_\tau})=d(E_{L,I},E_{I,L})$. If~$L$ is rotated, and chambers $x\in E_{L,I}, y\in E_{I,L}$ realise the distance $d(E_{L,I},E_{I,L})$, then $\tilde f (y)\in E_{I_\tau,L_\tau}$, and so by ($*$) and Lemma~\ref{lem:folding} we have $d(\E_{L_\tau,I_\tau},\E_{I_\tau,L_\tau})\leq d(\tilde f(x),\tilde f (y))<d(E_{L,I},E_{I,L})$.

To summarise, if there is a big component that is not self-compatible, then there is maximal spherical~$L$ that is rotated and maximal spherical~$I$ that contains $\{s,t\}$, hence not rotated. If all big components are self-compatible, and there is a small component incompatible with another component, then one of them is twisted and another is idle, so there is maximal spherical $L$ that is twisted  and maximal spherical $I$ that is idle with $I_1\neq \{s,t\}$. In both situations we obtain $\mathcal K_2(S)<\mathcal K_2(\tau(S))$, which is a contradiction.
\end{proof}

\section{Making use of consistent doubles}
\label{sec:consistent}

In this section we prove Theorem~\ref{thm:minimal complexity}, which as pointed out in Section~\ref{sec:proof} implies the Main Theorem.

\begin{lem}
\label{lem:double_side}
Suppose that $S$ has consistent doubles.
Let $L\subset S$ be irreducible spherical and let $r\in S$ with $L\cup \{r\}$ not spherical. Consider non-commuting $s,t\in L$ with $\{s,t,r\}$ not spherical.
Then $\Delta^{(s,t),r}$ does not depend on $s,t$, and we can denote it $\Delta^{L,r}$

Moreover, for $L$ exposed, for $C_L$ the vertex set of any cell of $\Da$ fixed by $L$, and for any chamber $x$ incident to $\mathcal W_r$, we have $d(C_L\cap \Delta^{L,r},x)<d(C_L\cap w_L\Delta^{L,r}, x)$.
\end{lem}
\begin{proof} To start we focus on the first assertion. Since doubles are consistent, by Remark~\ref{rem:moves_delta} we have $\Delta^{(s,t),r}=\Delta^{(s,L),r}$. Consider first the case where $|L|\geq 3$. Then by Lemma~\ref{lem:Emu} we have that $\Delta^{(s,L),r}$ does not depend on $s$ as long as $s$ is not a leaf of the Coxeter--Dynkin diagram of $L$. However, if $s$ is a leaf and $t$ is not a leaf, since the doubles are consistent observe that the pair $\Phi_s^{(s,t),r}, \Phi_t^{(t,s),r}$ is geometric. This implies $\Delta^{(s,t),r}=\Delta^{(t,s),r}$ and the assertion follows. In the case where $|L|=2$ it is enough to invoke that last observation.

For the second assertion, assume first that we have $|L|=2$ and that $V$ is the sector for $L$ containing $\Delta:=\Delta^{L,r}$. Then by Remark~\ref{rem:self-comp} we have $\mathcal W_r\subset sV\cup V\cup tV$. The required inequality follows then from e.g.\ Lemma~\ref{lem:folding} applied to one of the two foldings from the proof of Proposition~\ref{prop:bigcompatible}. If $|L|=3$, suppose that $s,t,p$ are consecutive vertices in the Coxeter--Dynkin diagram of $L$. Since $$\mathcal W_r\subset \Phi(t\mathcal W_s,\Delta)\cap \Phi(t\mathcal W_p,\Delta)\cap \Phi(s\mathcal W_t,\Delta)\cap \Phi(p\mathcal W_t,\Delta)\cap\Phi(ps\mathcal W_t,\Delta),$$ we have that $\mathcal W_r$ is contained in a sector for $L$ separated by at most two walls in $\mathcal W_L$ from $V$. Since $\mathcal W_L$ consists of at least $6$ walls separating $\Delta$ from $w_L\Delta$, the inequality follows.
\end{proof}

\begin{proof}[Proof of Theorem~\ref{thm:minimal complexity}]
By Proposition~\ref{prop:smallcompatible}, $S$ has consistent doubles.
By Corollary~\ref{cor:geometric criterion}, to prove Theorem~\ref{thm:minimal complexity} it suffices to show that for any simple markings $\mu$ and $\mu'$ with common core $s\in S$, we have $\Phi^{\mu}_s=\Phi^{\mu'}_s$. For each component $A$ of $S\setminus(s\cup s^{\perp})$, by Remark~\ref{rem:find markings} there exists a simple marking $\mu$ with core $s$ such that $K^\mu_s\subseteq A$ (where $K^\mu_s$ is as in Definition~\ref{def:K}). We now repeat the construction of halfspaces associated to components from
Definition~\ref{def:K}, with $\{s,t\}$ replaced by~$s$. Namely, since $S$ has consistent doubles, by Proposition~\ref{prop:same side2}, if $K_s^{\mu'}\subseteq A$, then $\Phi^{\mu}_s=\Phi^{\mu'}_s$. Thus each component $A$ of $S\setminus(s\cup s^{\perp})$ determines a halfspace $\Phi_s^{A}:=\Phi^{\mu}_s$ for $s$. Two components $A_1,A_2$ are \emph{compatible} if $\Phi_s^{A_1}=\Phi_s^{A_2}$. We will show that all components of $S\setminus(s\cup s^{\perp})$ are compatible.

Let $A$ be a component of $S\setminus(s\cup s^{\perp})$ and let $L\subset S$ be maximal spherical intersecting $A$. Note that if $s\notin L$, then there is $m\in L$ not adjacent to $s$ and hence using the marking $((s,\emptyset),m)$ we observe that
$C_{L}\subset \Phi_s^A$. Suppose now that $L$ contains $s$, and let $L_1\subset L$ be maximal irreducible containing $s$ and hence also containing some $t\in A$. Let $I\subset S$ be maximal spherical with some $r\in I\cap B$ for another component $B$ of $S\setminus(s\cup s^{\perp})$.
If $L_1$ is not exposed, then using the marking $\mu=((s,t),r)$, by the first assertion in Lemma~\ref{lem:double_side}, we have $\E_{L,I}\subseteq \Phi_s^A$. If $L_1$ is exposed, then by the second assertion in Lemma~\ref{lem:double_side}, for each
chamber~$y$ in $\E_{L,I}$ realising the distance to any fixed chamber of $\E_{I,L}$, we have $y\in \Phi_s^A$ as well.

If some components of $S\setminus(s\cup s^{\perp})$ are not compatible, let $A$ be the union of all components $A_i$ with one $\Phi_s^{A_i}$, and $B$ the union of components $B_i$ with the other $\Phi_s^{B_i}$. Let $\tau$ be the elementary twist that sends each element $b\in B$ to $sbs$, and fixes the other elements of~$S$.
Let $L,I\subset S$ be maximal spherical.
By the observation above on~$C_{L}$, we have $d(C_{L_\tau}, C_{I_\tau})\leq d(C_L, C_I)$ with strict inequality if and only if $s\notin L,I$ and $L\cap A_i,I\cap B_j\neq \emptyset$ or vice versa. Thus we can assume that such $L,I$ do not exist, and hence $\mathcal K_1(\tau(S))=\mathcal K_1(S)$ so that we can focus on~$\mathcal K_2$. Assume without loss of generality that all big components are among the $B_j$. Again from the above paragraph if $L\cap A_i,I\cap B_j\neq \emptyset$, then the chambers realising the distance between $\E_{L,I}$ and $\E_{I,L}$ lie in the opposite halfspaces for $s$. Futhermore, $\E_{I_\tau,L_\tau}=s\E_{I,L}$ by Remark~\ref{rem:L_isubsetI}(iii). 
Let $t\in L,r\in I$ be non-commuting with~$s$. Since $A_i$ is small, we have that $\{s,t\}$ is spherical. Since $\{s,t\}$ is self-compatible, by Remark~\ref{rem:self-comp} we have $\Phi(t\mathcal W_s,\mathcal W_r)=\Phi(t\mathcal W_s,s\mathcal W_r)$, and so $\E_{L_\tau,I_\tau}=\E_{L,I}$. Thus $d(\E_{L_\tau,I_\tau}, \E_{I_\tau,L_\tau})< d(\E_{L,I},\E_{I,L})$ and consequently, $\mathcal K_2(\tau(S))< \mathcal K_2(S)$, which is a contradiction.
\end{proof}

\end{document}